\documentclass[12pt]{amsart}
\usepackage{subfigure}
\usepackage{mathdots} 
\usepackage{tikz}
\usetikzlibrary{matrix,arrows,calc}
\usetikzlibrary{positioning}
\usepackage{makecell}
\usepackage{csquotes}
\usepackage{mathabx}
\usepackage{paralist}

\newcommand{\abs}[1]{\ensuremath{\left\vert#1\right\vert}}
\newcommand\restr[2]{{
  \left.\kern-\nulldelimiterspace 
  #1 
  \vphantom{|} 
  \right|_{#2} 
  }}

\definecolor{myblue}{cmyk}{1.00,0.56,0.00,0.34}
\definecolor{mygreen}{cmyk}{0.5,0,0.5,0.5}
\definecolor{myred}{cmyk}{0.00,1.00,0.63,0.00}
\definecolor{myyellow}{cmyk}{0.00,0.15,1.00,0.00}

\usepackage{chngcntr}
\counterwithin{figure}{section}
\usepackage{float} 

\usepackage{soul}

\usepackage{amssymb}
\usepackage{graphicx}
\usepackage{datetime}

\usepackage{geometry} 
\usepackage[all]{xy}

\usepackage{hyperref}  

\let\emptyset\varnothing  

\title[Infinite friezes]{Infinite friezes}

\author{Karin Baur, Mark J. Parsons, Manuela Tschabold}
\address{Institut f\"{u}r Mathematik und Wissenschaftliches Rechnen, 
Universit\"{a}t Graz, NAWI Graz, Heinrichstrasse 36, 
A-8010 Graz, Austria}
\email{baurk@uni-graz.at}
\email{mark.parsons@uni-graz.at}
\email{manuela.tschabold@uni-graz.at}

\newcommand{\monthword}[1]{\ifcase#1\or January\or February\or March\or April\or May\or June\or July\or August\or September\or October\or November\or December\fi}

\date{\monthword{\the\month} \the\day, \the\year } 

\begin{document}

\newtheorem{lm}{Lemma}[section]
\newtheorem{prop}[lm]{Proposition}
\newtheorem{satz}[lm]{Satz}

\newtheorem{theorem}[lm]{Theorem}
\newtheorem*{thm}{Theorem}
\newtheorem{conj}[lm]{Conjecture}

\newtheorem{cor}[lm]{Corollary}
\newtheorem*{corollary}{Corollary}

\theoremstyle{definition}
\newtheorem{defn}[lm]{Definition}
\newtheorem{ex}[lm]{Example}
\newtheorem*{rem}{Remark}
\newtheorem{remark}[lm]{Remark}


\newcommand{\perm}{\operatorname{Perm}\nolimits}
\newcommand{\NN}{\operatorname{\mathbb{N}}\nolimits}
\newcommand{\ZZ}{\operatorname{\mathbb{Z}}\nolimits}
\newcommand{\U}{\operatorname{\mathbb{U}}\nolimits}
\newcommand{\V}{\operatorname{\mathbb{V}}\nolimits}

\begin{abstract}
We provide a characterization of infinite frieze patterns of positive integers 
via triangulations of an infinite strip in the plane. In the periodic case, 
these triangulations may be considered as triangulations of annuli.  
We also give a geometric interpretation of all entries of 
infinite friezes via matching numbers.
\end{abstract}

\maketitle

%
\section*{Introduction}
%

Frieze patterns were introduced and studied by Coxeter in \cite{cox} and 
by Conway-Coxeter in \cite{coco1,coco2}. 
By definition, these are arrays of a finite number of infinite rows, shifted as in the picture below. 
A frieze pattern starts with a row of zeros followed by a 
row of ones and ends with a row of ones followed by a row of zeros. 
All entries in between are positive integers and satisfy 
the unimodular rule, i.e.\ for every diamond with entries $a,b,c,d$ as below, 
we have $ad=bc+1$. Henceforth, we will usually refer to these as \emph{finite friezes}. 

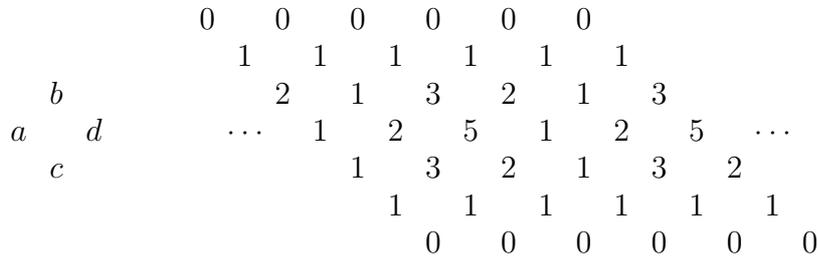
\begin{figure}[H]
\begin{tikzpicture}[font=\normalsize]

  \matrix(m) [matrix of math nodes,row sep={1.2em,between origins},column sep={1.2em,between origins},nodes in empty cells]{
&&&&&0&&0&&0&&0&&0&&0&&&&&&\\
&&&&&&1&&1&&1&&1&&1&&1&&&&&\\
&b&&&&&&2&&1&&3&&2&&1&&3&&&&\\
a&&d&&&&\node{\cdots};&&1&&2&&5&&1&&2&&5&&\node{\cdots};&\\
&c&&&&&&&&1&&3&&2&&1&&3&&2&&\\
&&&&&&&&&&1&&1&&1&&1&&1&&1&\\
&&&&&&&&&&&0&&0&&0&&0&&0&&0\\
};
   
\end{tikzpicture}
\end{figure}

Such friezes are always invariant under a glide reflection and hence periodic. 
In \cite{coco1,coco2}, Conway and Coxeter gave a characterization of finite friezes by 
triangulations of polygons, hence providing a geometric interpretation for them: 
Given a triangulated $n$-gon, let $a_i$ be the number of triangles incident with 
vertex $i$, providing an $n$-tuple 
$(a_1,\dots, a_n)$. 
Then, if we fill the first non-trivial row with infinitely many repeated copies of this 
$n$-tuple, and use the unimodular 
rule, we obtain a finite frieze with $n+1$ rows. 
Conversely, every finite frieze with $n+1$ rows arises in this way. 

Broline, Crowe and Isaacs extended this in \cite{bci} by giving a geometric 
interpretation for all rows using matching numbers for the triangulations. 

Around 2000, there was renewed interest in this topic stemming from the
appearance of frieze patterns in the context of cluster algebras of type A, \cite{cach,propp}. 
Frieze patterns have also been shown to appear in the context of cluster algebras of type D, \cite{bm}.

Frieze patterns and similar patterns have also been studied 
by various authors in the 
last decade, eg.\ \cite{hj}, \cite{ars}, \cite{mgot}, to mention just a few. Recent work \cite{mg} 
by Morier-Genoud gives an excellent overview and points out many different directions 
of research. 

In this article, we study infinite friezes. These are arrays consisting of infinitely many rows of 
positive integers, bounded above by a row of zeros followed by a row of ones, satisfying the unimodular 
rule everywhere. The third author has shown (in \cite{tschabold}) that triangulations of punctured discs 
give rise to periodic infinite friezes and that these satisfy a beautiful arithmetic property. 
Furthermore, \cite{tschabold} 
gives these friezes a geometric interpretation by realizing the entries as matching numbers, and also 
as \enquote{labels} for triangulations (cf.\ \cite[Answer 32]{coco2}). 
The periodic infinite friezes considered 
in \cite{tschabold} can be 
viewed as arising from cluster algebras in 
type $\tilde{A}$, with the exchange relation arising from the unimodular rule. 

It is clear that there are infinite friezes that do not arise from triangulations of punctured discs (see 
\cite[Proposition 2.9]{tschabold}, for example). In this article, we prove that all infinite friezes can be obtained via triangulations of surfaces. In 
the periodic case, infinite friezes arise from triangulations of annuli. In the non-periodic case, we can 
use triangulations of an infinite strip in the plane to obtain the friezes. 

SL$_2$-tilings of the plane were studied in \cite{ars, hj}. These are bi-infinite 
lattices of positive integers and hence are not friezes in our set-up. 
In \cite{hj}, the authors show that some of these tilings 
can be obtained by triangulations of a strip, while in \cite{bhj}, it is explained 
how all SL$_2$-tilings can be obtained from triangulations of a circle with a 
discrete set of marked points 
as well as four accumulation points on the boundary. 
Such triangulations have been introduced by Igusa and Todorov in \cite{it}. 

This article is structured as follows. In Section~\ref{sec:prelims}, 
we recall the definition of infinite friezes and some relations between the entries in friezes. 
In Section~\ref{sec:sequences}, 
we show that increasing entries in the first non-trivial row of an infinite frieze 
results again in an infinite frieze (Theorem~\ref{thm:modifyquiddityrow}). 
After this, we recall in Section~\ref{sec:friezes-triangulations} how 
triangulations of discs and of punctured discs give rise to finite and infinite friezes (respectively) and 
prove that triangulations of annuli 
give rise to infinite friezes
(Theorem~\ref{thm:annulus-frieze}). Furthermore, we 
explain how to go between triangulations of punctured discs and annuli 
(with no marked points on the inner boundary), using asymptotic triangulations. 
In Section~\ref{sec:classification}, we show that every periodic infinite frieze can 
be obtained from a triangulation of an annulus (Theorem~\ref{thm:periodic-combinatorial}). 
The proof provides an algorithm 
for checking whether an infinite periodic sequence of positive integers gives rise to 
a finite frieze, an infinite frieze, or neither.

Finally, in Section~\ref{sec:final-matchings}, we turn our attention to 
triangulations of an infinite strip in the plane, providing a geometric 
interpretation of all entries in an infinite frieze. 
We show that every (admissible) triangulation of an infinite strip 
gives rise to an infinite frieze (Theorem~\ref{thm:arbitrary-matchings}) and 
that every infinite frieze can be realized in this way 
(Theorem~\ref{thm:friezes-realize}). 
Moreover, we show that the entries of an infinite frieze can be interpreted as matching numbers 
between vertices and triangles of any associated triangulation. 
In particular, this result covers periodic infinite friezes (which arise from annuli and hence also 
from \enquote{periodic triangulations} of an infinite strip), and thus extends 
\cite[Theorem 5.21]{tschabold} 
as well as generalizing the aforementioned result of \cite{bci}.

%
\section{Preliminaries}\label{sec:prelims}
%

\begin{defn}
An \emph{infinite frieze} $\mathcal{F}$ of positive integers is an array 
$(m_{ij})_{i,j\in\mathbb{Z}, j\ge i-2}$ of shifted infinite rows such that 
$m_{i,i-2}=0$, $m_{i,i-1}=1$ and $m_{ij}\in\mathbb{Z}_{>0}$ for all $i\leq j$
\begin{figure}[H]
\begin{tikzpicture}[font=\normalsize] 

\matrix(m) [matrix of math nodes,row sep={1.75em,between origins},column sep={1.75em,between origins},nodes in empty cells]{
&0&&0&&0&&0&&0&&&&&\\[-0.25em]
\node{\cdots};&&1&&1&&1&&1&&1&&\node{\cdots};&&&\\[-0.25em]
&&&m_{-1,-1}&&m_{00}&&m_{11}&&m_{22}&&m_{33}&&\\
&&\node{\cdots}; &&m_{-1,0}&&m_{01}&&m_{12}&&m_{23}&& m_{34} &&\node{\cdots};&\\
&&&&&m_{-1,1}&&m_{02}&&m_{13}&&m_{24}&&m_{35} \\
&&&&&&&&\node[rotate=-6.5,shift={(-0.034cm,-0.08cm)}]  {\ddots};&&&&\node[rotate=-6.5,shift={(-0.034cm,-0.08cm)}]  {\ddots};&&&\\
};
     
\end{tikzpicture}
\end{figure}
\noindent 
satisfying 
the \emph{unimodular rule}, i.e.\ 
for every diamond in $\mathcal{F}$ of the form
\begin{figure}[H]
\begin{tikzpicture}[font=\normalsize] 

\matrix(m) [matrix of math nodes,row sep={1.75em,between origins},column sep={1.75em,between origins},nodes in empty cells]{
&m_{i+1,j}&\\
m_{ij}&&m_{i+1,j+1}\\
&m_{i,j+1}&\\
     };
     
\end{tikzpicture}
\end{figure} 

\noindent 
the relation $m_{ij}m_{i+1,j+1}-m_{i+1,j}m_{i,j+1}=1$ holds. 
\end{defn}

We use $(i,j)$ to indicate the position of the entry $m_{ij}$. 
The first non-trivial row of such a frieze is called the {\em quiddity row}. 
It will often be denoted by $(a_i)_{i\in\mathbb{Z}}$, with $a_i=m_{ii}$ for all $i$. 
It will be convenient to extend the infinite frieze by setting 
$m_{ij}=0$ for all $(i,j)$ with $i>j+2$. 

 Note that quiddity means \enquote{essence}, and since the first non-trivial row determines the frieze together 
with the row of $1$'s above it, it makes sense to call $(a_i)_i$ the quiddity row.

\begin{defn}\label{def:period}
If the quiddity row $(a_i)_i$ of an infinite frieze $\mathcal{F}$ has period $n$, we say that the 
frieze has period $n$. 
The $n$-tuple $(a_1,a_2,\dots,a_n)$ is a {\em quiddity sequence} for the frieze. 
If $n$ is minimal such that $\mathcal{F}$ is $n$-periodic, we say that $n$ is the 
{\em shortest period} of $\mathcal{F}$. 
\end{defn}

Clearly, the quiddity sequence of an $n$-periodic infinite frieze is only determined up to cyclic equivalence. Any $n$ consecutive diagonals of an $n$-periodic infinite frieze 
describe the whole frieze completely. 

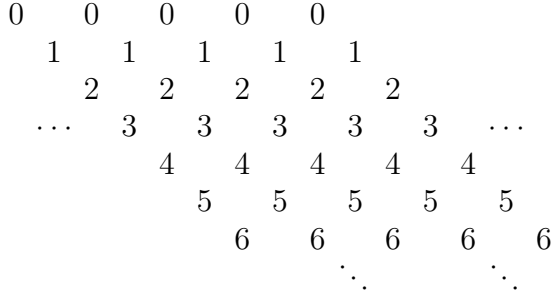
\begin{figure}[t]
\begin{tikzpicture}[font=\normalsize] 

  \matrix(m) [matrix of math nodes,row sep={1.2em,between origins},column sep={1.2em,between origins},nodes in empty cells]{
0&&0&&0&&0&&0&&&&&&\\
&1&&1&&1&&1&&1&&&&&\\
&&2&&2&&2&&2&&2&&&&\\
&\node{\cdots};&&3&&3&&3&&3&&3&&\node{\cdots};&\\
&&&&4&&4&&4&&4&&4&&\\
&&&&&5&&5&&5&&5&&5&\\
&&&&&&6&&6&&6&&6&&6\\
&&&&&&&&&\node[rotate=-6.5,shift={(-0.034cm,-0.08cm)}]  {\ddots};&&&&\node[rotate=-6.5,shift={(-0.034cm,-0.08cm)}]  {\ddots};&\\
};

\end{tikzpicture}
\caption{The infinite frieze with constant quiddity row $a_i=2$.}\label{fig:basicfrieze}
\end{figure}

Figure \ref{fig:basicfrieze} gives an example of a periodic infinite frieze 
$(m_{ij})_{i,j}$. 
All the rows are constant, $m_{ij}=j-i+2$.

It is well known how the entries of a frieze and its quiddity sequence depend on 
each other and we will use these relations several times throughout this article. 
On the one hand, we have
\begin{align}\label{friezerelationa}
m_{ij}=\det \begin{pmatrix}a_i&1&&&0\\
1&a_{i+1}&1&\\
&\ddots&\ddots&\ddots\\
&&1&a_{j-1}&1\\
0&&&1&a_{j}
\end{pmatrix},
\end{align}
 and on the other hand,
\begin{align}\label{friezerelationb}
m_{ij}=a_j m_{i,j-1}-m_{i,j-2}, \  \mbox{ and} \quad m_{ij}=a_i m_{i+1,j}-m_{i+2,j}.
\end{align}

The following lemma gives a further useful relation between the entries in a frieze, which we will make use of in Section~\ref{sec:sequences}.

\begin{lm}\label{friezerelationc}
Let $\mathcal{F}=(m_{ij})_{i,j}$ be a frieze. Then
$$
m_{ij}=m_{i,k-1}m_{kj}-m_{i,k-2}m_{k+1,j}
$$
for every $i\leq k\leq j+1$.
\end{lm}

\begin{proof}
Induction on $j-i=r$. If $i-j=0$, then $k=i$ or $k=i+1$. In both cases, the first term on the right hand 
side is $m_{ij}$ and the second term is $0$. 

Consider $j-i=1$. Then $k\in\{i,i+1,i+2\}$. If $k=i$, the right hand side is 
$m_{i,i-1}m_{i,i+1}-m_{i,i-2}m_{i+1,i+1}$ and since $m_{i,i-1}=1$, $m_{i,i-2}=0$, 
the claim follows. The case $k=i+2$ is completely analogous.

Let $k=i+1$. Then the right hand side is 
$m_{ii}m_{i+1,i+1}-m_{i,i-1}m_{i+2,i+1}=m_{ii}m_{i+1,i+1}-1$, and by the 
unimodular rule, the latter is $m_{i,i+1}m_{i+1,i}=m_{i,i+1}$. 

So assume the claim holds for $j-i\le r$ and consider $j-i=r+1$. We have
$m_{ij} = m_{ii}m_{i+1,j}-m_{i+2,j}$ 
(by equation (\ref{friezerelationb})). 
The terms on 
the right hand side have subscripts with difference at most $r$, so the induction hypothesis 
can be applied to replace $m_{i+1,j}$ and $m_{i+2,j}$: 
\begin{eqnarray*}
m_{ij} & = & m_{ii}m_{i+1,j}-m_{i+2,j} \\ 
  & = & m_{ii}(m_{i+1,k-1}m_{kj} - m_{i+1,k-2}m_{k+1,j}) - m_{i+2,k-1}m_{kj} + m_{i+2,k-2}m_{k+1,j}
\end{eqnarray*}
for $i+2\le k\le j+1$. By (\ref{friezerelationb}), $m_{ii}m_{i+1,k-1}=m_{i,k-1}+m_{i+2,k-1}$ and 
$m_{ii}m_{i+1,k-2}=m_{i,k-2}+m_{i+2,k-2}$, and so 
\begin{eqnarray*}
m_{ij} & = & (m_{i,k-1}+m_{i+2,k-1})m_{kj} - (m_{i,k-2}+m_{i+2,k-2})m_{k+1,j} \\ 
   & &  - m_{i+2,k-1}m_{kj}  + m_{i+2,k-2}m_{k+1,j}  \\Ê
   & = & m_{i,k-1}m_{kj} - m_{i,k-2}m_{k+1,j}.
\end{eqnarray*}
The case $k=i$ is clear. For $k=i+1$, we have $m_{ij} = m_{ii}m_{i+1,j}-m_{i+2,j}$ by (\ref{friezerelationb}),
and the result follows immediately since $m_{i,i-1}=1$. 
\end{proof}

%
\section{Quiddity rows}\label{sec:sequences}
%

The next result shows how we can obtain new infinite friezes from known ones. 

\begin{theorem}\label{thm:modifyquiddityrow}
Let $\mathcal{F}$ be an infinite frieze with quiddity row $(a_i)_{i\in\mathbb{Z}}$. Let $k$ 
be a fixed integer. Then for every $b\in\mathbb{Z}_{>0}$, 
the sequence $(a'_i)_{i\in\mathbb{Z}}$ defined by
$$
a'_i=\begin{cases}
a_i+b &\text{if }i=k,\\
a_i &\text{otherwise},
\end{cases}
$$
is the quiddity row of an infinite frieze. 
At position $(i,j)$, the entry of this frieze is 
\begin{equation*} 
m_{ij} + bm_{i,k-1}m_{k+1,j}. 
\end{equation*}
In particular, outside the cone with peak at position $(k,k)$, the entries 
of the new frieze coincide with those of $\mathcal{F}$. 
\end{theorem}

\begin{proof}
We define a lattice $(m'_{ij})_{i,j}$ of (positive) integers whose first non-trivial 
row is $(a_i')_i$ and we show 
that the unimodular rule is satisfied everywhere. 
Let 
$$
m'_{ij}=m_{ij} + bm_{i,k-1}m_{k+1,j}
$$
For $i\le j+1$, the $m_{ij}'$ are all positive. For $k<i$ or $j<k$, 
$m_{i,k-1}m_{k+1,j}=0$ and hence the unimodular rule holds outside the 
downward cone with peak at position $(k,k)$. 

To show that the unimodular rule holds within the cone, we have to show 
$m'_{ij}m'_{i+1,j+1}-m'_{i+1,j}m'_{i,j+1}=1$ for $i+1\le k\le j$. With a single
application of the unimodular rule inside $\mathcal{F}$, we see that

\begin{align*}
&m'_{ij}m'_{i+1,j+1}-m'_{i+1,j}m'_{i,j+1}=\left(m_{ij} + bm_{i,k-1}m_{k+1,j}\right)\left(m_{i+1,j+1}+ bm_{i+1,k-1}m_{k+1,j+1}\right)\\
&\ \ \ -\left(m_{i+1,j} + bm_{i+1,k-1}m_{k+1,j}\right)\left(m_{i,j+1} + bm_{i,k-1}m_{k+1,j+1}\right)\\
&= m_{ij}m_{i+1,j+1}+bm_{ij}m_{i+1,k-1}m_{k+1,j+1}+bm_{i+1,j+1}m_{i,k-1}m_{k+1,j}\\
&\ \ \ -m_{i+1,j}m_{i,j+1}-bm_{i+1,j}m_{i,k-1}m_{k+1,j+1}-bm_{i,j+1}m_{i+1,k-1}m_{k+1,j}\\
&=1+bm_{i+1,k-1}(m_{ij} m_{k+1,j+1}
-\ m_{i,j+1} m_{k+1,j})\\
&\ \ \ +bm_{i,k-1}( m_{i+1,j+1}m_{k+1,j}
- m_{i+1,j} m_{k+1,j+1}).
\end{align*}
Using Lemma~\ref{friezerelationc} to replace $m_{ij}$, $m_{i,j+1}$, $m_{i+1,j+1}$ and $m_{i+1,j}$, 
it then follows that $m'_{ij}m'_{i+1,j+1}-m'_{i+1,j}m'_{i,j+1}=1$.

It remains to check diamonds linking the boundaries of the cone with the region outside. On 
the left side, these are diamonds involving 
$m'_{i-1,k-1}, m'_{ik}$ for $i \le k$: 
\begin{align*}
m'_{i-1,k-1}m'_{ik}-m'_{i,k-1}m'_{i-1,k}&=
m_{i-1,k-1}(m_{ik} + bm_{i,k-1}m_{k+1,k})\\
&\ \ \ -m_{i,k-1}(m_{i-1,k} + bm_{i-1,k-1}m_{k+1,k})\\
&=m_{i-1,k-1}m_{ik}-m_{i,k-1}m_{i-1,k}\\
&=1.
\end{align*}
The ones on the right side involve $m'_{kj}$ and $m'_{k+1,j+1}$ for $j \ge k$, and an analogous computation 
shows $m'_{kj}m'_{k+1,j+1}-m'_{k+1,j}m'_{k,j+1}=1$. 
\end{proof}

Within the infinite frieze, the change of the entry at position $(i,j)$ is determined by the intersection of 
the cone with peak at $(k+1,k-1)$ and the cone with lowest point at position $(i,j)$: 
If this intersection is a rectangle, the two 
corners different from the peaks are multiplied and the result is added $b$ 
times to $m_{ij}$. 
Example ~\ref{ex:increase-entry} below illustrates this. 

\begin{cor}\label{cor:all-2-ok}
Every infinite sequence $(a_i)_{i}$ with $a_i\ge 2$ for all $i\in\mathbb{Z}$ defines an infinite 
frieze of quiddity row $(a_i)_i$. 
\end{cor}

\begin{ex}\label{ex:increase-entry}
We consider the infinite frieze from Figure \ref{fig:basicfrieze} with $a_i=2$ for all $i$ 
and replace one entry $a_k$ by $3$. 
The result is the quiddity row $(\cdots ~ 2 ~ 2 ~ 3 ~ 2 ~ 2 ~ \cdots)$ with infinite 
frieze as shown here. The cone with peak at position $(k,k)$ appears as shaded. 
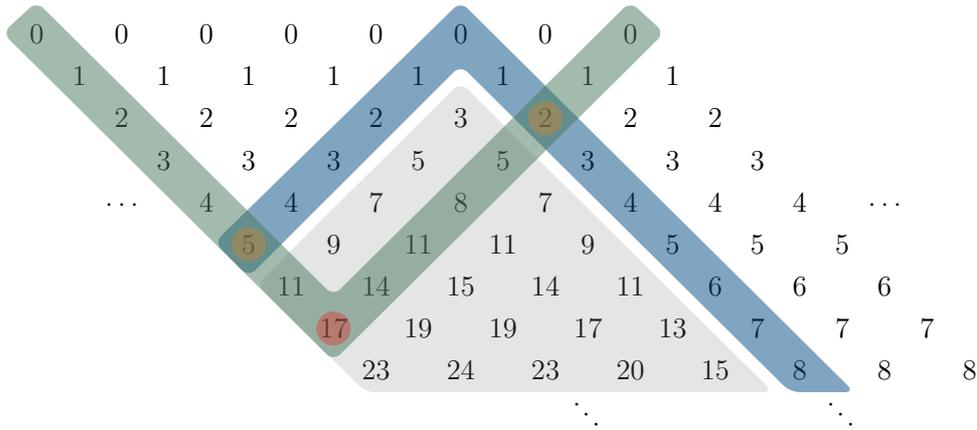
\begin{figure}[H]
\scalebox{.9}{\begin{tikzpicture}[font=\normalsize] 

\matrix(m) [matrix of math nodes,row sep={1.5em,between origins},column sep={1.5em,between origins},nodes in empty cells]{
0&&0&&0&&0&&0&&0&&0&&0&&&&&&&&\\
&1&&1&&1&&1&&1&&1&&1&&1&&&&&&&\\
&&2&&2&&2&&2&&3&&2&&2&&2&&&&&&\\
&&&3&&3&&3&&5&&5&&3&&3&&3&&&&&\\
&&\node{\cdots};&&4&&4&&7&&8&&7&&4&&4&&4&&\node{\cdots};&&\\
&&&&&5&&9&&11&&11&&9&&5&&5&&5&&&\\
&&&&&&11&&14&&15&&14&&11&&6&&6&&6&&\\
&&&&&&&17&&19&&19&&17&&13&&7&&7&&7&\\
&&&&&&&&23&&24&&23&&20&&15&&8&&8&&8\\
&&&&&&&&&&&&&\node[rotate=-6.5,shift={(-0.034cm,-0.08cm)}]  {\ddots};&&&&&&\node[rotate=-6.5,shift={(-0.034cm,-0.08cm)}]  {\ddots};&&&\\
};

\draw[rounded corners,opacity=0,fill=gray,fill opacity=0.2] ($(m-7-7.north)+(-0.5cm,-0.25cm)$) -- ($(m-3-11.north)+(0cm,0.25cm)$) --($(m-9-17.north)+(0.85cm,-0.6cm)$) -- ($(m-9-9.north)+(-0.15cm,-0.6cm)$) -- cycle;

\draw[rounded corners,opacity=0,fill=myblue,fill opacity=0.5] ($(m-6-6.north)+(-0.5cm,-0.25cm)$) -- ($(m-1-11.north)+(0cm,0.2cm)$) -- ($(m-9-19.north)+(0.8cm,-0.6cm)$) -- ($(m-9-17.north)+(1.1cm,-0.6cm)$) -- ($(m-3-11.north)+(0cm,0.5cm)$)  -- ($(m-6-6.south)+(0cm,-0.2cm)$) -- cycle;

\draw[rounded corners,opacity=0,fill=mygreen,fill opacity=0.5] ($(m-1-1.north)+(0cm,0.2cm)$) -- ($(m-8-8.north)+(0cm,0.2cm)$)-- ($(m-1-15.north)+(0cm,0.2cm)$) -- ($(m-1-15.south)+(0.5cm,0.3cm)$)--($(m-8-8.south)+(0cm,-0.2cm)$)-- ($(m-1-1.south)+(-0.5cm,0.3cm)$) -- cycle;

\fill[orange,opacity=0.3] (m-6-6) circle (0.25cm);
\fill[orange,opacity=0.3] (m-3-13) circle (0.25cm);
\fill[red,opacity=0.3] (m-8-8) circle (0.25cm);



\end{tikzpicture}}
\end{figure}
\noindent 
The entry at position $(k-4,k+1)$ is $17$ (in red). It is obtained by adding the product 
of the (orange) entries at positions $(k-4,k-1)$ and $(k+1,k+1)$ to the initial value $7$. 
\end{ex}

\begin{rem} 
If the initial infinite frieze is periodic, 
we can apply Theorem~\ref{thm:modifyquiddityrow} iteratedly, 
to obtain a new periodic infinite frieze. By the remarks after the theorem, this will result in accumulated 
increases in the entries of the frieze, as the regions where entries become larger, intersect infinitely often. 
\end{rem}
 
%
\section{Periodic friezes from triangulations}\label{sec:friezes-triangulations}
%

Triangulations of surfaces are a source of finite and periodic infinite 
friezes. We will show that every periodic 
frieze can be obtained from a surface (Theorem~\ref{thm:periodic-combinatorial}). 
It is a well-known result of Conway and 
Coxeter that every finite frieze arises from a triangulation of a polygon. 
In the infinite case, the surfaces are punctured discs \cite{tschabold}, and annuli. 
In all these cases, quiddity sequences may be read-off directly from the triangulations. We begin by briefly recalling results of \cite{coco1,coco2,tschabold}. 

\subsection{Finite friezes}
Consider a triangulation of a polygon with $n$ vertices, labeled anticlockwise around the 
boundary. Let $a_i$ be the number of matchings between vertex $i$ and triangles of the 
triangulation, i.e.\ the number of triangles incident with vertex $i$. 
Then $(a_i)_{1\le i\le n}$ 
is the quiddity sequence of an $n$-periodic finite frieze with $n+1$ rows (including 
the first and last rows of zeros). 
Such a frieze is invariant under a glide reflection. 
The finite friezes are precisely those arising in this way from triangulations of polygons. 

\subsection{Friezes from punctured discs}
In a similar way, we can associate friezes to triangulations of punctured discs. Let 
$S_n$ be a punctured disc with $n$ marked points on its boundary and a puncture $0$ 
in the interior. We label the marked points $\{1,2,\dots, n\}$ anticlockwise around 
the boundary. 
{\em Arcs} in $S_n$ are non-contractible 
curves whose endpoints are marked points from $\{1,2,\dots, n\}\cup\{0\}$, 
considered up to isotopy fixing endpoints. 
Boundary segments are not considered 
to be arcs. 
We only allow curves that do not 
have self-intersections (though the endpoints may coincide). Arc between a marked point 
on the boundary and the puncture are called {\em central}, all other arcs are {\em peripheral}. 
Two arcs are {\em compatible} if there are 
representatives in their isotopy classes that do not intersect. 
A {\em triangulation} of $S_n$ is a maximal collection of pairwise compatible arcs in 
$S_n$. Every triangulation contains $n$ arcs, and 
induces a subdivision of $S_n$ into three-sided regions. 
Such a triangulation may include a self-folded triangle, i.e.\ a triangular region formed by 
a central arc from a marked point $i$ on the boundary 
to the puncture together with a loop at $i$. 

To a triangulation $T$, we associate an $n$-tuple of positive integers as follows: 
Let $S_n\setminus \{\alpha\mid \alpha \in T\}$ be the collection of open regions inside 
$S_n$. Then for $i=1,\dots, n$, we define $a_i$ to be the number of regions locally visible 
in a small neighborhood around $i$. 
The $n$-tuple $(a_1,\dots, a_n)$ is called the quiddity sequence of the triangulation. 

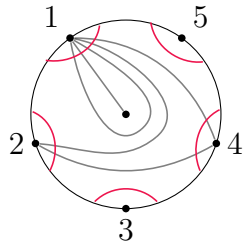
\begin{figure}[H]
\begin{tikzpicture}[scale=1,font=\normalsize] 

\node (a) at (0,0) [fill,circle,inner sep=1pt] {};
\draw (0,0) circle (1.25cm);
\foreach \x in {-144,-72,0,72,144} {
   \begin{scope}[rotate=\x]
    \node (\x) at (0,-1.25) [fill,circle,inner sep=1pt] {};
   \end{scope}
}
\draw (144) node [above right] {$5$};
\draw (72) node [right] {$4$};
\draw (0) node [below] {$3$};
\draw (-72) node [left] {$2$};
\draw (-144) node [above left] {$1$};

\draw[semithick,opacity=0.5] (a) to (-144);
\draw[semithick,opacity=0.5,out=-140,in=-80] (0.2,-0.2) to (-144);
\draw[semithick,opacity=0.5,out=-30,in=40] (-144) to  (0.2,-0.2);
\draw[semithick,opacity=0.5,out=110,in=0] (72) to  (-144);
\draw[semithick,opacity=0.5,out=-30,in=-150] (-72) to (72);
\draw[semithick,opacity=0.5,out=-5,in=-140] (-72) to (0.4,-0.3) ;
\draw[semithick,opacity=0.5,out=40,in=-10] (0.4,-0.3) to(-144);

\foreach \x in {-144,-72,0,72,144} {
   \begin{scope}[rotate=\x]

\foreach \y in {-20,20} {
   \begin{scope}[rotate=\y]
    \node (\y) at (0,-1.25) [circle,inner sep=0pt] {};
   \end{scope}
}
\draw[semithick,myred,out=45,in=135] (-20) to (20);
   \end{scope}
}

\draw (5,0) node {$(a_i)_{1\le i\le 5}=(6,3,1,3,1)$};

\end{tikzpicture}
\end{figure}

\begin{remark}
If we view a triangulation of $S_n$ as a graph with $n+1$ vertices 
(the marked points) 
and $2n$ edges, then we have $a_i=\deg i -1$ for all $1\le i\le n$. 
\end{remark} 

\begin{theorem}[{{\cite[Theorem 4.6]{tschabold}}}]\label{thm:triangul-frieze-Sn} 
Fix a triangulation of $S_n$ and let $a_i$ be as defined above. 
Then $(a_1,\dots, a_n)$ is the quiddity sequence of a periodic infinite frieze. 
Its shortest period is a divisor of $n$. 
\end{theorem}

\subsection{Friezes from annuli} 

For $n,m$ $\in \mathbb{Z}_{\ge 0}$, with $n>0$, we write 
$A_{n,m}$ for an annulus with $n$ marked points on the outer boundary 
and $m$ marked points on the inner boundary. We will denote the marked 
points on the outer boundary by $1,2,\dots, n$ and the marked points on the 
inner boundary by $n+1,\dots, n+m$. 

It is customary to view $A_{n,m}$ as a cylinder of height $1$. We can cut this 
cylinder open to lie in the plane where the two vertical boundaries are 
identified,  
with lower boundary for 
the outer boundary, upper boundary for the inner boundary of the annulus. 
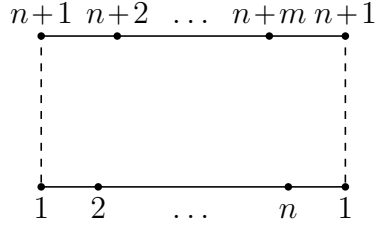
\begin{figure}[H]
\begin{tikzpicture}[scale=1,font=\normalsize]

\draw[semithick] (0,1) to (4,1);

\foreach \x in {0,1,3,4} {
   \begin{scope}[shift={(\x cm, 0 cm)}]
    \node (\x) at (0,1) [fill,circle,inner sep=1pt] {};
   \end{scope}
}
\draw (0) node [above] {$n\!+\!1$};
\draw (1) node [above] {$n\!+\!2$};
\node at (2cm,1.032cm) [above] {$\dots$};
\draw (3) node [above] {$n\!+\!m$};
\draw (4) node [above] {$n\!+\!1$};

\draw[semithick] (0,-1) to (4,-1);

\node (0) at (0,-1) [fill,circle,inner sep=1pt] {};
\node (1) at (0.75,-1) [fill,circle,inner sep=1pt] {};
\node (3) at (3.25,-1) [fill,circle,inner sep=1pt] {};
\node (4) at (4,-1) [fill,circle,inner sep=1pt] {};

\draw (0) node [below] {$1$};
\draw (1) node [below] {$2$};
\node at (2cm,-1.231cm) [below] {$\dots$};
\draw (3) node [below,shift={(0,-0.09)}] {$n$};
\draw (4) node [below] {$1$};

\draw[semithick,dashed] (0) to (0,1);
\draw[semithick,dashed] (4) to (4,1);

\end{tikzpicture}
\end{figure}
It will be convenient to work in the universal cover $\U=\U(n,m)$ 
of the cylinder and hence of 
$A_{n,m}$, viewing $\U$ as an infinite strip in $\mathbb{R}^2$ of height 
$1$. Whereas marked points of $A_{n,m}$ are denoted by the numbers 
$\{1,2,\dots, n+m\}$, the marked points in the universal cover 
are denoted by coordinates in the plane: 
If $m\ne 0$, the marked points are 
$$
\{(km,0)\mid k\in \mathbb{Z}\}\  \cup \ \{(kn,1)\mid k\in \mathbb{Z}\}. 
$$
For $m=0$, the marked points are 
$$
\{(k,0)\mid k\in \mathbb{Z}\}. 
$$

The covering map $\pi$ from $\U$ to the cylinder (and hence to $A_{n,m}$) 
is induced from wrapping $\U$ around the cylinder. If $m>0$, 
its effect on a boundary vertex is to reduce its first coordinate modulo 
$nm$ and then divide it by $m$ for the lower boundary (resp.\ $n$ for the upper boundary)
and then increase it by 1 (resp.\ $n+1$):
$$
\begin{array}{ll}
\U\ni (r+knm,0)  \mapsto r/m +1 \in A_{n,m} & \mbox{ for } 0\le r< nm ,\\
\U\ni (s+knm,1)  \mapsto s/n+ n+1  \in A_{n,m} & \mbox{ for } 0\le s<  nm. 
\end{array}
$$
If $m=0$, the first coordinate is reduced modulo $n$ and then increased by 1: 
$$
\begin{array}{ll}
\U\ni (t+kn,0)  \mapsto t+1 \in A_{n,0} & \mbox{ for } 0\le t< n. 
\end{array}
$$

\begin{figure}[h]
\begin{tikzpicture}[scale=1.25,font=\normalsize] 

\draw (-0.5,0) node  [shift={(-0.25,0)}] {$\dots$};
\draw (9.5,0) node [shift={(0.25,0)}] {$\dots$};

\draw[opacity=0,fill=gray,fill opacity=0.15] (0,-1) -- (0,1) --  (9,1) -- (9,-1)--cycle;

\draw[semithick] (0,1) to (9,1);

\foreach \x in {0,3,6,9} {
   \begin{scope}[shift={(\x cm, 0 cm)}]
    \node (\x) at (0,1) [fill,circle,inner sep=1pt] {};
   \end{scope}
}
\draw (0) node [above] {$(-6,1)$};
\draw (3) node [above] {$(0,1)$};
\draw (6) node [above] {$(6,1)$};
\draw (9) node [above] {$(12,1)$};

\foreach \x in {0,3,6} {
   \begin{scope}[shift={(\x cm, 0 cm)}]
    \node (\x) at (1.5,1) [fill,circle,inner sep=1pt] {};
   \end{scope}
}
\draw (0) node [above] {$(-3,1)$};
\draw (3) node [above] {$(3,1)$};
\draw (6) node [above] {$(9,1)$};

\draw[semithick] (0,-1) to (9,-1);

\foreach \x in {0,1,2,3,4,5,6,7,8,9} {
   \begin{scope}[shift={(\x cm, 0 cm)}]
    \node (\x) at (0,-1) [fill,circle,inner sep=1pt] {};
   \end{scope}
}

\draw (0) node [below] {$(-6,0)$};
\draw (1) node [below] {$(-4,0)$};
\draw (2) node [below] {$(-2,0)$};
\draw (3) node [below] {$(0,0)$};
\draw (4) node [below] {$(2,0)$};
\draw (5) node [below] {$(4,0)$};
\draw (6) node [below] {$(6,0)$};
\draw (7) node [below] {$(8,0)$};
\draw (8) node [below] {$(10,0)$};
\draw (9) node [below] {$(12,0)$};

\draw[semithick,dashed] (0) to (0,1);
\draw[semithick,dashed] (3) to (3,1);
\draw[semithick,dashed] (6) to (6,1);
\draw[semithick,dashed] (9) to (9,1);

\end{tikzpicture}
\caption{Universal cover for $n=3$, $m=2$.}\label{fig:cover}
\end{figure}
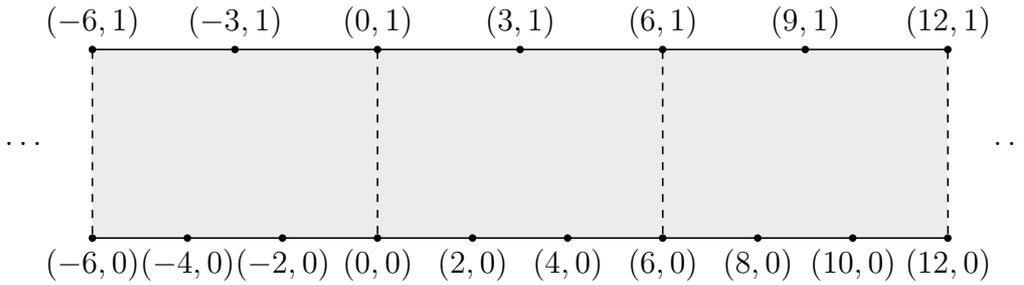

The rectangle with vertices $(0,0), (mn,0),(mn,1),(0,1)$ is a {\em fundamental 
domain} for $(\U,\pi)=(\U(n,m),\pi)$. Any rectangle in $\U$ of the same size 
is also a fundamental domain for $\pi$. The fundamental domains we consider 
are usually rectangles whose vertices are marked points. 
The universal cover with several choices of a fundamental 
domain is illustrated for $A_{3,2}$ in Figure~\ref{fig:cover}. 

We use arcs in $A_{n,m}$ to triangulate the annulus. We define {\em (finite) arcs} 
as isotopy classes of curves just as for the punctured disc. 
Finite arcs with both endpoints on the same boundary are called 
peripheral. An arc is {\em bridging} otherwise. 
We denote by $z$ a non-contractible closed curve in the interior 
of $A_{n,m}$ and call it the {\em meridian}. 
Then we extend the set-up by introducing 
asymptotic arcs as follows. 

\begin{defn}
An {\em asymptotic arc} is a curve that has one endpoint at 
a marked point and winds around the meridian infinitely (considered 
up to isotopy fixing its endpoint). 
An asymptotic arc with endpoint $i$ is called the  
{\em Pr\"ufer arc} $\pi_i$ at $i$ if it spirals around $z$ positively, 
and the {\em adic arc} $\alpha_i$ at $i$ otherwise.
\end{defn}

Asymptotic arcs have been used in the context of categories of tube types to  
extend maximal rigid objects to cluster-tilting objects, cf.\ \cite{bm-tube}. 
These arcs correspond to indecomposable objects that arise through directed systems.

\smallskip

The set of 
arcs of $A_{n,m}$ is the collection of finite arcs together 
with the asymptotic arcs of the annulus. Two 
arcs are said to be {\em compatible}, if there are representatives in their isotopy 
classes that do not intersect.  

\begin{defn}
A {\em triangulation} of $A_{n,m}$ is a maximal collection of pairwise 
compatible arcs. 
\end{defn}

In the figure below, two triangulations of $A_{3,2}$ are given. The triangulation 
on the left contains Pr\"ufer and adic arcs. In the picture, we have used the notation 
$[i,j]$ to describe peripheral arcs. It will also 
be convenient to use this notation for 
bridging arcs, e.g.\ writing $[1,4]$ for $\gamma$. 

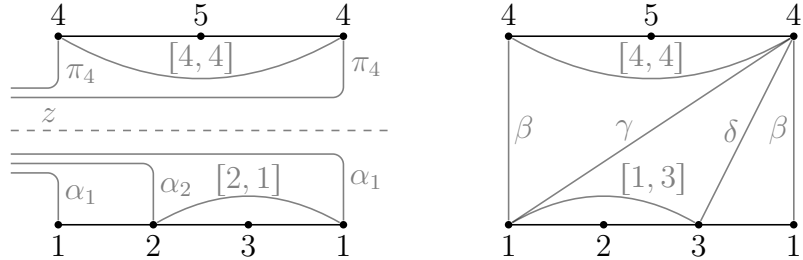
\begin{figure}[H]
\begin{tikzpicture}[scale=1.25,font=\normalsize] 

\draw[semithick] (0,1) to (3,1);

\node (a) at (0,1) [fill,circle,inner sep=1pt] {};
\node (b) at (1.5,1) [fill,circle,inner sep=1pt] {};
\node (c) at (3,1) [fill,circle,inner sep=1pt] {};

\draw (a) node [above] {$4$};
\draw (b) node [above] {$5$};
\draw (c) node [above] {$4$};

\draw[semithick,opacity=0.5,out=-30,in=-150] (a) to node [sloped,shift={(0,0.27)}] {$[4,4]$} (c);

\draw[rounded corners,semithick,opacity=0.5] (a) to (0,0.45) to (-0.5,.45);
\draw[opacity=0.5] (a) node [shift={(0.3,-0.5)}] {$\pi_4$};

\draw[rounded corners,semithick,opacity=0.5] (c) to (3,0.35) to (-0.5,.35);
\draw[opacity=0.5] (c) node [shift={(0.3,-0.4)}] {$\pi_4$};

\draw[semithick,opacity=0.5,dashed] (-0.5,0) to node [shift={(-2,0.2)}] {$z$} (3.5,0);

\draw[semithick] (0,-1) to (3,-1);

\foreach \x in {0,1,2,3} {
   \begin{scope}[shift={(\x cm, 0 cm)}]
    \node (\x) at (0,-1) [fill,circle,inner sep=1pt] {};
   \end{scope}
}
\draw (0) node [below] {$1$};
\draw (1) node [below] {$2$};
\draw (2) node [below] {$3$};
\draw (3) node [below] {$1$};

\draw[semithick,opacity=0.5,out=30,in=150] (1) to node [sloped,shift={(0,0.22)}] {$[2,1]$} (3);

\draw[rounded corners,semithick,opacity=0.5] (3) to  (3,-0.25) to (-0.5,-0.25);
\draw[opacity=0.5] (0) node [shift={(0.3,0.4)}] {$\alpha_1$};

\draw[rounded corners,semithick,opacity=0.5] (1) to (1,-0.35) to (-0.5,-0.35);
\draw[opacity=0.5] (1) node [shift={(0.3,0.5)}] {$\alpha_2$};

\draw[rounded corners,semithick,opacity=0.5] (0) to (0,-0.45) to (-0.5,-0.45);
\draw[opacity=0.5] (3) node [shift={(0.3,0.6)}] {$\alpha_1$};

\end{tikzpicture}
\hskip 1cm
\begin{tikzpicture}[scale=1.25,font=\normalsize] 

\draw[semithick] (0,1) to (3,1);

\node (a) at (0,1) [fill,circle,inner sep=1pt] {};
\node (b) at (1.5,1) [fill,circle,inner sep=1pt] {};
\node (c) at (3,1) [fill,circle,inner sep=1pt] {};

\draw (a) node [above] {$4$};
\draw (b) node [above] {$5$};
\draw (c) node [above] {$4$};

\draw[semithick] (0,-1) to (3,-1);

\foreach \x in {0,1,2,3} {
   \begin{scope}[shift={(\x cm, 0 cm)}]
    \node (\x) at (0,-1) [fill,circle,inner sep=1pt] {};
   \end{scope}
}
\draw (0) node [below] {$1$};
\draw (1) node [below] {$2$};
\draw (2) node [below] {$3$};
\draw (3) node [below] {$1$};

\draw[semithick,opacity=0.5] (0) to node [shift={(0.2,0)}] {$\beta$} (a);
\draw[semithick,opacity=0.5] (0) to node [shift={(-0.35,0)}] {$\gamma$} (c);
\draw[semithick,opacity=0.5] (2) to node [shift={(-0.2,0)}] {$\delta$} (c);
\draw[semithick,opacity=0.5] (3) to node [shift={(-0.2,0)}] {$\beta$} (c);

\draw[semithick,opacity=0.5,out=30,in=150] (0) to node [sloped,shift={(0.65,0.225)}] {$[1,3]$} (2);
\draw[semithick,opacity=0.5,out=-30,in=-150] (a) to node [sloped,shift={(0,0.27)}] {$[4,4]$} (c);

\end{tikzpicture}
\end{figure}

Triangulations involving asymptotic arcs were introduced in~\cite{bdupont}. 
Every triangulation of $A_{n,m}$ contains $n+m$ arcs (\cite[Corollary 1.8]{bdupont}). 
Triangulations do not need to involve asymptotic arcs. Since 
bridging arcs and  asymptotic arcs are never compatible, 
every triangulation either has bridging arcs or asymptotic arcs. 
We will say that a triangulation is {\em asymptotic} if it uses 
Pr\"ufer or adic arcs. 

\begin{rem}
We are mainly interested in matching numbers between the marked points and 
the triangles of a given  triangulation. If a triangulation contains arcs 
with large winding numbers, we can unwind these, as the numbers of triangles incident 
with a vertex are not affected by winding or unwinding all arcs of a triangulation 
simultaneously. So we can always choose triangulations to be ``simple'' in the sense that 
they avoid large winding numbers. 
\end{rem}

Let $T$ be a triangulation of $A_{n,m}$ and consider its preimage $\pi^{-1}(T)$ 
in $\U$. This is a triangulation of $\U$, consisting of repeated copies of 
a fundamental domain. 
For $1\le i\le  n$, 
we let $a_i$ be the number of triangles incident with vertex 
$((i-1)m,0)\in \U$ if $m>0$ and with $(i-1,0)$ if $m=0$.  
We call $(a_1,\dots, a_n)$ the {\em (outer) quiddity sequence} of $T$. 
This notion is motivated by the fact that the 
sequence $(a_1,\dots, a_n)$ gives rise to an infinite frieze (cf.\ Theorem~\ref{thm:annulus-frieze}). 

\begin{remark}\label{rem:subsequencs}
Note that the quiddity sequence $(a_1,\dots, a_n)$ of a triangulation might 
contain periodically repeating subsequences. Namely, there may exist a non-trivial factor 
$1<r<n$ of $n$, say $n=rs$, 
such that $(a_1,\dots, a_n)$ consists of 
$s$ successive copies of $(a_1,\dots, a_{r})$. 
\end{remark}

\begin{lm}\label{lm:S_n-annulus}
\begin{inparaenum}[$(i)$]
\item Let $T$ be a triangulation of $S_n$ with quiddity sequence $(a_1,\dots, a_n)$. 
Then for every $m\ge 0$, there exist 
triangulations of $A_{n,m}$ with 
quiddity sequence $(a_1,\dots, a_n)$.

\item Let $T$ be a 
triangulation of $A_{n,0}$ with quiddity 
sequence $(a_1,\dots, a_n)$. Then there exists a triangulation of $S_n$ with 
this quiddity sequence. 
\end{inparaenum}
\end{lm}

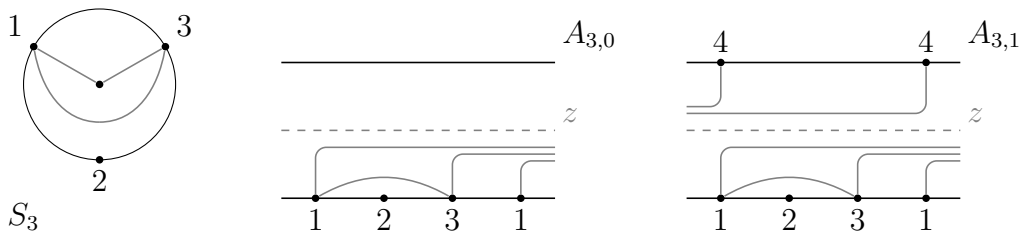
\begin{figure}[H]
\begin{tikzpicture}[font=\normalsize] 

\draw (0) node [shift={(-1,-0.5)}] {$S_3$};

\node (a) at (0,0) [fill,circle,inner sep=1pt] {};
\draw (0,0) circle (1cm);
\foreach \x in {-120,0,120} {
   \begin{scope}[rotate=\x]
    \node (\x) at (0,-1) [fill,circle,inner sep=1pt] {};
   \end{scope}
}
\draw (120) node [above right] {$3$};
\draw (0) node [below] {$2$};
\draw (-120) node [above left] {$1$};

\draw[semithick,opacity=0.5] (a) to (120);
\draw[semithick,opacity=0.5] (a) to (-120);

\draw[semithick,opacity=0.5,out=180,in=-80] (0,-0.5) to  (-120);
\draw[semithick,opacity=0.5,out=-100,in=0] (120) to  (0,-0.5);

\end{tikzpicture}
\hskip .75cm
\begin{tikzpicture}[scale=0.9,font=\normalsize] 

\draw (4,1) node [above] {$A_{3,0}$};

\draw[semithick] (-0.5,1) to (3.5,1);

\draw[semithick,opacity=0.5,dashed] (-0.5,0) to node [shift={(2,0.2)}] {$z$} (3.5,0);

\draw[semithick] (-0.5,-1) to (3.5,-1);

\foreach \x in {0,1,2,3} {
   \begin{scope}[shift={(\x cm, 0 cm)}]
    \node (\x) at (0,-1) [fill,circle,inner sep=1pt] {};
   \end{scope}
}
\draw (0) node [below] {$1$};
\draw (1) node [below] {$2$};
\draw (2) node [below] {$3$};
\draw (3) node [below] {$1$};

\draw[semithick,opacity=0.5,out=30,in=150] (0) to (2);

\draw[rounded corners,semithick,opacity=0.5] (3) to  (3,-0.45) to (3.5,-0.45);
\draw[rounded corners,semithick,opacity=0.5] (2) to (2,-0.35) to (3.5,-0.35);
\draw[rounded corners,semithick,opacity=0.5] (0) to (0,-0.25) to (3.5,-0.25);

\end{tikzpicture}
\hskip .5cm
\begin{tikzpicture}[scale=0.9,font=\normalsize] 

\draw (4,1) node [above] {$A_{3,1}$};

\draw[semithick,opacity=0.5,dashed] (-0.5,0) to node [shift={(2,0.2)}] {$z$} (3.5,0);

\draw[semithick] (-0.5,1) to (3.5,1);

\node (a) at (0,1) [fill,circle,inner sep=1pt] {};
\node (b) at (3,1) [fill,circle,inner sep=1pt] {};

\draw (a) node [above] {$4$};
\draw (b) node [above] {$4$};

\draw[rounded corners,semithick,opacity=0.5] (a) to  (0,0.35) to (-0.5,0.35);
\draw[rounded corners,semithick,opacity=0.5] (b) to (3,0.25) to (-0.5,0.25);

\draw[semithick] (-0.5,-1) to (3.5,-1);

\foreach \x in {0,1,2,3} {
   \begin{scope}[shift={(\x cm, 0 cm)}]
    \node (\x) at (0,-1) [fill,circle,inner sep=1pt] {};
   \end{scope}
}
\draw (0) node [below] {$1$};
\draw (1) node [below] {$2$};
\draw (2) node [below] {$3$};
\draw (3) node [below] {$1$};

\draw[semithick,opacity=0.5,out=30,in=150] (0) to (2);

\draw[rounded corners,semithick,opacity=0.5] (3) to  (3,-0.45) to (3.5,-0.45);
\draw[rounded corners,semithick,opacity=0.5] (2) to (2,-0.35) to (3.5,-0.35);
\draw[rounded corners,semithick,opacity=0.5] (0) to (0,-0.25) to (3.5,-0.25);

\end{tikzpicture}
\end{figure}

\begin{proof}
\begin{inparaenum}[(i)] 
\item To go from triangulations of $S_n$ to asymptotic triangulations of $A_{n,m}$, every peripheral 
arc of $S_n$ is kept while every central arc is replaced by a Pr\"ufer arc (or by an adic arc)
based at the outer boundary. 
The result of this is a collection of $n$ compatible arcs based at the outer boundary. 
For $m=0$, this is a triangulation of 
$A_{n,0}$. 
In case $m>0$, to obtain a triangulation of $A_{n,m}$, one adds $m$ compatible
(peripheral and asymptotic) arcs based at the inner boundary (thus creating an
asymptotic triangulation at the inner boundary) to the set of arcs obtained as above. 
There are many choices for this. 

\item One keeps peripheral arcs, 
while replacing every asymptotic arc of $T$ by a central arc based at the corresponding
marked point of $S_n$. 
\end{inparaenum}
\end{proof}

By Lemma~\ref{lm:S_n-annulus}, every quiddity sequence of $A_{n,0}$ gives rise 
to an (periodic) infinite frieze. Since triangulations of $S_n$ can be viewed 
as triangulations of $A_{n,0}$, we may work in $\U=\U(n,0)$ to determine the matching 
numbers $a_i$. 

\begin{theorem}\label{thm:annulus-frieze}
Let $T$ be a triangulation of $A_{n,m}$ with quiddity sequence $(a_1,\dots, a_n)$. 
Then $(a_1,\dots, a_n)$ defines a periodic infinite frieze. 
\end{theorem}

\begin{proof} 
If the triangulation is asymptotic, we can 
consider the arcs based at the outer boundary separately. They form 
a triangulation of $A_{n,0}$ and the claim follows from Theorem \ref{thm:triangul-frieze-Sn} and 
Lemma~\ref{lm:S_n-annulus}.

Assume that the triangulation $T$ of $A_{n,m}$ does not contain any 
asymptotic arcs. 
Hence it contains at least $2$ bridging arcs 
(\cite[Lemma 1.7]{bdupont}).
We associate an asymptotic 
triangulation $T'$ to $T$ as follows. 
The arcs in $T'$ are the peripheral arcs of $T$ together with 
the adic arc $\alpha_i$ for every marked point $i$ of $A_{n,m}$ incident with a bridging arc 
of $T$, as illustrated in the figure. 
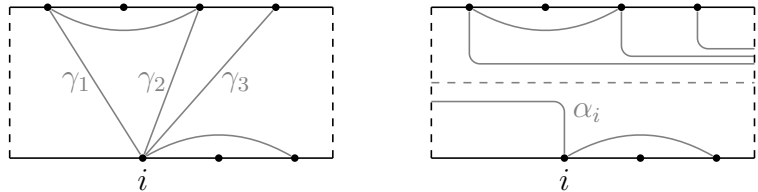
\begin{figure}[H]
\begin{tikzpicture}[scale=1,font=\normalsize] 

\draw[semithick,dashed] (0.25,-1) to (0.25,1);
\draw[semithick,dashed] (4.5,-1) to (4.5,1);

\draw[semithick] (0.25,1) to (4.5,1);

\node (a) at (0.75,1) [fill,circle,inner sep=1pt] {};
\node (b) at (1.75,1) [fill,circle,inner sep=1pt] {};
\node (c) at (2.75,1) [fill,circle,inner sep=1pt] {};
\node (d) at (3.75,1) [fill,circle,inner sep=1pt] {};

\draw[semithick] (0.25,-1) to (4.5,-1);

\foreach \x in {2,3,4} {
   \begin{scope}[shift={(\x cm, 0 cm)}]
    \node (\x) at (0,-1) [fill,circle,inner sep=1pt] {};
   \end{scope}
}

\draw (2) node [below] {$i$};

\draw[semithick,opacity=0.5] (2) to node [shift={(-0.25,0)}] {$\gamma_1$} (a);
\draw[semithick,opacity=0.5] (2) to node [shift={(-0.25,0)}] {$\gamma_2$}  (c);
\draw[semithick,opacity=0.5] (2) to node [shift={(0.35,0)}] {$\gamma_3$} (d);
\draw[semithick,opacity=0.5,out=30,in=150] (2) to (4);
\draw[semithick,opacity=0.5,out=-30,in=-150] (a) to (c);

\end{tikzpicture}
\hskip 1cm
\begin{tikzpicture}[scale=1,font=\normalsize] 

\draw[semithick,dashed] (0.25,-1) to (0.25,1);
\draw[semithick,dashed] (4.5,-1) to (4.5,1);

\draw[semithick,opacity=0.5,dashed] (0.25,0) to (4.5,0);

\draw[semithick] (0.25,1) to (4.5,1);

\node (a) at (0.75,1) [fill,circle,inner sep=1pt] {};
\node (b) at (1.75,1) [fill,circle,inner sep=1pt] {};
\node (c) at (2.75,1) [fill,circle,inner sep=1pt] {};
\node (d) at (3.75,1) [fill,circle,inner sep=1pt] {};

\draw[semithick,opacity=0.5,out=-30,in=-150] (a) to (c);
\draw[rounded corners,semithick,opacity=0.5] (a) to  (0.75,0.25) to (4.5,0.25);
\draw[rounded corners,semithick,opacity=0.5] (c) to (2.75,0.35) to (4.5,0.35);
\draw[rounded corners,semithick,opacity=0.5] (d) to (3.75,0.45) to (4.5,0.45);

\draw[semithick] (0.25,-1) to (4.5,-1);

\foreach \x in {2,3,4} {
   \begin{scope}[shift={(\x cm, 0 cm)}]
    \node (\x) at (0,-1) [fill,circle,inner sep=1pt] {};
   \end{scope}
}

\draw (2) node [below] {$i$};

\draw[rounded corners,semithick,opacity=0.5] (2) to (2,-0.25) to (0.25,-0.25);
\draw[opacity=0.5] (2) node [shift={(0.3,0.6)}] {$\alpha_i$};

\draw[semithick,opacity=0.5,out=30,in=150] (2) to (4);

\end{tikzpicture}
\end{figure}
\noindent
This induces a triangulation of $A_{n,0}$. 
Let $(b_1,\dots, b_n)$ be its quiddity sequence. It defines 
a periodic infinite frieze 
(Theorem~\ref{thm:triangul-frieze-Sn}, Lemma~\ref{lm:S_n-annulus}). 
We claim 
that $b_i\le a_i$ holds for all $i$. 
If $i$ is a marked point not incident with any bridging arc of $T$, then 
$b_i=a_i$. If $i$ is incident with $r_i$ bridging arcs of $T$, then in $T'$, 
$i$ is incident with one adic arc. Hence $b_i=a_i - (r_i-1)\le a_i$. 
The claim then follows using Theorem~\ref{thm:modifyquiddityrow}. 
\end{proof}

Theorem~\ref{thm:annulus-frieze} motivates the following definition: 

\begin{defn}\label{def:realizable}
Let $(a_1,\dots, a_n)$ be the quiddity sequence of a periodic frieze. 
We say that $(a_1,\dots, a_n)$ can be {\em realized (in an annulus)} 
if there is some $m\ge0$ and a triangulation $T$ of $A_{n,m}$ such that 
$(a_1,\dots, a_n)$ is the quiddity sequence of $T$. 
\end{defn}

%
\section{Classification of periodic friezes}\label{sec:classification}
%

The main goal of this section is to prove the following claim:
Let $(a_1,\dots, a_n)$ be the quiddity sequence of an $n$-periodic frieze $\mathcal F$. 
Then either $\mathcal F$ is a finite frieze and hence $(a_1,\dots, a_n)$ is the quiddity 
sequence of a triangulated polygon (\cite{coco1,coco2}) or $\mathcal F$ is infinite 
and we can realize $(a_1,\dots, a_n)$ in an annulus.

Observe that in the former 
case, up to rotating triangulations, 
there is only one triangulated polygon giving rise to $(a_1,\dots, a_n)$ and this 
polygon may have more than $n$ vertices 
(see (a), (b) and (c) of Example~\ref{ex:small-period} below). 
Whereas, in the latter case, we 
can realize $(a_1,\dots, a_n)$ in an annulus with $s\cdot n$ marked points on the outer boundary 
for every 
$s\ge 1$ (Theorem~\ref{thm:periodic-combinatorial} and Lemma~\ref{lm:period-multiply}). 

\begin{ex}\label{ex:small-period}
Let $\mathcal F$ be the $n$-periodic frieze with quiddity sequence $(a_1,\dots, a_n)$. 

\begin{inparaenum}[(a)]
\item If $n=1$ and the quiddity sequence is $(1)$, then 
$\mathcal F$ is finite and arises from the trivial triangulation of a triangle.

\item For $n=2$ with quiddity sequence $(1,2)$, the frieze is finite and arises 
from a triangulation of a square.

\item If $n=2$ and the quiddity sequence is $(1,3)$,  
$\mathcal F$ is finite and arises from a triangulation of a hexagon by an inner triangle.

\item If $\mathcal F$ is the $2$-periodic frieze with quiddity sequence $(1,4)$, then 
$\mathcal F$ is infinite and arises from a triangulation of $A_{2,0}$ with a loop and one asymptotic 
arc. This is equivalent to the sequence obtained from a triangulation of $S_2$ with one 
loop and one central arc.

\item If $\mathcal F$ is a $2$-periodic frieze with quiddity sequence $(1,a)$ with $a\ge 5$, then 
the frieze is infinite and arises from a triangulation of $A_{2,n-4}$ with one peripheral 
arc (a loop) and $a-3$ bridging arcs.
\end{inparaenum}

\end{ex}

We first note that if an infinite frieze arises from a triangulation 
of an annulus, it arises from infinitely many triangulations of different annuli: 

\begin{lm}\label{lm:period-multiply}
Let $(a_1,\dots, a_n)$ be the quiddity sequence of a triangulation of $A_{n,m}$. Then 
for every $s\ge 1$, there exists a triangulation of $A_{sn,sm}$ with quiddity sequence 
$$
(\underbrace{a_1,\dots, a_n}_{\tiny\mbox{first copy}},
\underbrace{a_1,\dots, a_n}_{\tiny\mbox{second}},\dots,
\underbrace{a_1,\dots, a_n}_{\tiny\mbox{$s^{th}$ copy}}).
$$ 
\end{lm}

\begin{proof}
Let $T$ be a triangulation of $A_{n,m}$ and consider its 
preimage $\pi^{-1}(T)$ in $\U$. 
We take a fundamental domain for $A_{n,m}$ and $s-1$ further copies of 
it to the right. This large rectangle can be considered as a fundamental domain 
for the annulus $A_{sn,sm}$.
\end{proof}

When determining whether a periodic infinite frieze with quiddity 
sequence $(a_i)_{1\le i\le n}$ is 
realizable in an annulus, it is sufficient to find a realization for 
any periodic subsequence $(a_1,\dots, a_r)$ of it (Remark~\ref{rem:subsequencs}). 
By Lemma~\ref{lm:period-multiply}, $(a_1,\dots, a_n)$ is then also realizable. 
It is thus enough to consider 
the shortest period $1\le r\le n$ of the frieze and to show that the 
sequence $(a_1,\dots, a_r)$ is realizable. 

\medskip

We next show how increasing entries in a quiddity sequence coming from a triangulation of 
an annulus gives rise to new realizable quiddity sequences. 
This can be considered as a combinatorial analogue of Theorem~\ref{thm:modifyquiddityrow}. 

\begin{prop}\label{prop:adding-one}
Let $(a_1,\dots, a_n)$ be realizable in $A_{n,m}$, $n>0$, $m\ge 0$. 
Let $1\le j\le n$ and consider the sequence 
$(\tilde{a}_1,\dots, \tilde{a}_n)$ 
$$
\mbox{ with $\tilde{a}_i:=$} 
     \left\{
     \begin{array}{ll}
     a_i & \mbox{if $i\ne j$}, \\
     a_i+1 & \mbox{if $i=j$}.     
     \end{array}
     \right.
$$
Then there exists $m'\ge m$ such that the sequence $(\tilde{a}_i)_i$ 
is realizable in $A_{n,m'}$. 
\end{prop}

\begin{proof}
Since $(a_i)_i$ is realizable in $A_{n,m}$ we pick a triangulation $T$ with 
quiddity sequence $(a_1,\dots, a_n)$. 
We distinguish two cases: 
\begin{inparaenum}[(A)]
\item there are no bridging and no asymptotic 
arcs at the marked point $j$ on the outer boundary, and 
\item there is a bridging or an asymptotic arc at $j$. 
\end{inparaenum}

\begin{inparaenum}[\text{In case} (A),]
\item there are $k$ peripheral arcs lying above $j$. 
We work in a fundamental domain of $\U$ with triangulation obtained from $\pi^{-1}(T)$.
We write these $k$ arcs as $[m_{i},l_{i}]$, $i=1,\dots,k$, where $m_{i},l_{i}$ are vertices on the lower boundary of $\U$ with 
\begin{equation} \label{eq:endpoints}
m_k\le m_{k-1} \le \dots\le m_1<j< l_1\le l_2\le\dots \le l_k,
\end{equation}
also using $j$ to denote the vertex $(j-1,0)$ in $\U$. 
Observe that we can never simultaneously have $m_{i+1}=m_i$ and $l_{i+1}=l_i$.

Denote the leftmost bridging arc at $m_k$ by $\alpha$ and the rightmost bridging arc at $l_k$ by $\beta$ 
(possibly, $\alpha=\beta$). 
Consider the region $P(j)$ bounded by $\alpha$ and $\beta$ (which contains the $k$ peripheral arcs above $j$).
This region is a polygon inside a fundamental domain for $A_{n,m}$. 

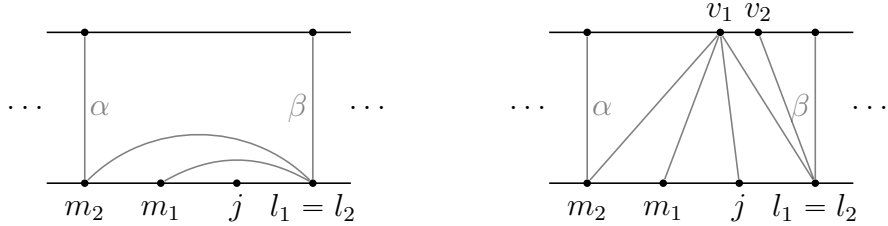
\begin{figure}[H]
\begin{tikzpicture}[scale=1,font=\normalsize] 

\draw (-0.5,0) node [shift={(-0.25,0)}] {$\dots$};
\draw (3.5,0) node [shift={(0.25,0)}] {$\dots$};

\draw[semithick] (-0.5,1) to (3.5,1);

\node (a) at (0,1) [fill,circle,inner sep=1pt] {};
\node (b) at (3,1) [fill,circle,inner sep=1pt] {};

\draw[semithick] (-0.5,-1) to (3.5,-1);

\foreach \x in {0,1,2,3} {
   \begin{scope}[shift={(\x cm, 0 cm)}]
    \node (\x) at (0,-1) [fill,circle,inner sep=1pt] {};
   \end{scope}
}

\draw (0) node [below,shift={(0,-0.095)}] {$m_2$};
\draw (1) node [below,shift={(0,-0.095)}] {$m_1$};
\draw (2) node [below] {$j$};
\draw (3) node [below] {$l_1=l_2$};

\draw[semithick,opacity=0.5,out=30,in=150] (1) to (3);
\draw[semithick,opacity=0.5,out=45,in=135] (0) to (3);
\draw[semithick,opacity=0.5] (0) to node [shift={(0.2,0)}] {$\alpha$} (a);
\draw[semithick,opacity=0.5] (3) to node  [shift={(-0.2,0)}] {$\beta$}  (b);

\end{tikzpicture}
\hskip 1cm
\begin{tikzpicture}[scale=1,font=\normalsize] 

\draw (-0.5,0) node  [shift={(-0.25,0)}] {$\dots$};
\draw (3.5,0) node [shift={(0.25,0)}] {$\dots$};

\draw[semithick] (-0.5,1) to (3.5,1);

\node (a) at (0,1) [fill,circle,inner sep=1pt] {};
\node (b) at (1.75,1) [fill,circle,inner sep=1pt] {};
\node (c) at (2.25,1) [fill,circle,inner sep=1pt] {};
\node (d) at (3,1) [fill,circle,inner sep=1pt] {};

\draw (b) node [above] {$v_1$};
\draw (c) node [above] {$v_2$};

\draw[semithick] (-0.5,-1) to (3.5,-1);

\foreach \x in {0,1,2,3} {
   \begin{scope}[shift={(\x cm, 0 cm)}]
    \node (\x) at (0,-1) [fill,circle,inner sep=1pt] {};
   \end{scope}
}

\draw (0) node [below,shift={(0,-0.095)}] {$m_2$};
\draw (1) node [below,shift={(0,-0.095)}] {$m_1$};
\draw (2) node [below] {$j$};
\draw (3) node [below] {$l_1=l_2$};

\draw[semithick,opacity=0.5] (0) to node [shift={(0.2,0)}] {$\alpha$} (a);
\draw[semithick,opacity=0.5] (3) to node  [shift={(-0.2,0)}] {$\beta$}  (d);

\draw[semithick,opacity=0.5] (0) to (b);
\draw[semithick,opacity=0.5] (1) to (b);
\draw[semithick,opacity=0.5] (2) to (b);
\draw[semithick,opacity=0.5] (3) to (b);
\draw[semithick,opacity=0.5] (3) to (c);

\end{tikzpicture}
\end{figure}

We now replace each peripheral arc above $j$ in $P(j)$ by a pair of bridging arcs,
and keep all other arcs. To do this, 
we need to increase the number of vertices on the inner boundary.  

We add a vertex 
$v_1$ above $j$. This changes the number of marked points, but we still call the region 
$P(j)$. 
In $P(j)$, we draw three (bridging) arcs connecting 
$m_1, j$ and $l_1$ with the new vertex. 

If $m_2>m_1$ and $l_2<l_1$, we connect $m_2$ and $l_2$ with $v_1$. 
If $m_2=m_1$ (resp.\ $l_2=l_1$), we draw a second vertex $v_2$ to the left (resp.\ right) of $v_1$,
and connect $m_2$ (resp.\ $l_2$) instead with $v_2$. 
We continue in this way until we have replaced the $k$ peripheral arcs above $j$ 
by $2k$ bridging arcs. Let $m'$ be the new number of vertices on the inner boundary of 
the fundamental domain.  
The number of triangles incident with the vertices $m_k,\dots, m_1$ and $l_1,\dots, l_k$ 
remains unchanged. The only change is at vertex $j$, where we have inserted a bridging 
arc, hence increasing the number of incident triangles by one. 
The resulting triangulation of $A_{n,m'}$ has the desired quiddity sequence. 

\medskip

\item[Case (B):]
\begin{inparaenum}[(i)]
\item If there are at least two bridging arcs incident with $j$, we consider a triangle based at 
$j$ formed by two bridging arcs $\alpha$ and $\beta$ 
at $j$ and a peripheral arc $\gamma$ or boundary segment $[r,r+1]$ on the 
inner boundary, $n+1\le r$, $r+1\le n+m$. 
In the former situation, we work in $A_{n,m}$ and 
replace the peripheral arc $\gamma$ by its flip: 
$\gamma$ 
is a diagonal in a quadrilateral formed by the two bridging arcs $\alpha$, $\beta$ 
and by two 
peripheral arcs or boundary segments. The flip replaces $\gamma$ by the other 
diagonal $\gamma'$ in the quadrilateral. 
Whereas, in the latter situation, we insert a new vertex $v_1$ between 
$r$ and $r+1$ and thus work in $A_{n,m+1}$. We 
add a new bridging arc connecting $j$ and $v_1$ inside the 
quadrilateral 
region formed by $\alpha$ and $\beta$ and the two boundary segments. 
In both cases, there are now $a_j+1$ triangles incident with $j$. 

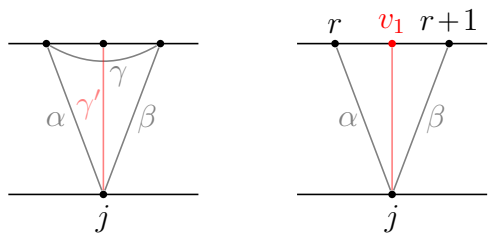
\begin{figure}[H]
\begin{tikzpicture}[scale=1,font=\normalsize] 

\draw[semithick] (-0.25,1) to (2.25,1);

\node (a) at (0.25,1) [fill,circle,inner sep=1pt] {};
\node (b) at (1,1) [fill,circle,inner sep=1pt] {};
\node (c) at (1.75,1) [fill,circle,inner sep=1pt] {};

\draw[semithick] (-0.25,-1) to (2.25,-1);
\node (j) at (1,-1) [fill,circle,inner sep=1pt] {};
\draw (j) node [below] {$j$};

\draw[semithick,opacity=0.5,out=-30,in=-150] (a) to node [shift={(0.2,-0.2)}] {$\gamma$} (c);
\draw[semithick,opacity=0.5] (j) to node [shift={(-0.25,0)}] {$\alpha$}  (a);
\draw[semithick,red,opacity=0.5] (j) to node [shift={(-0.19,0.2)}] {$\gamma'$} (b);
\draw[semithick,opacity=0.5] (j) to node [shift={(0.2,0)}] {$\beta$}  (c);

\end{tikzpicture}
\hskip 1cm
\begin{tikzpicture}[scale=1,font=\normalsize] 

\draw[semithick] (-0.25,1) to (2.25,1);

\node (a) at (0.25,1) [fill,circle,inner sep=1pt] {};
\node (b) at (1,1) [red,fill,circle,inner sep=1pt] {};
\node (c) at (1.75,1) [fill,circle,inner sep=1pt] {};

\draw (a) node [above] {$r$};
\draw[red] (b) node [above] {$v_1$};
\draw (c) node [above] {$r\!+\!1$};

\draw[semithick] (-0.25,-1) to (2.25,-1);

\node (j) at (1,-1) [fill,circle,inner sep=1pt] {};
\draw (j) node [below] {$j$};

\draw[semithick,opacity=0.5] (j) to node [shift={(-0.2,0)}] {$\alpha$}  (a);
\draw[semithick,red,opacity=0.5] (j) to (b);
\draw[semithick,opacity=0.5] (j) to node [shift={(0.2,0)}] {$\beta$}  (c);

\end{tikzpicture}
\end{figure}

\item If there is only one bridging arc $\alpha$ at $j$, there is a triangle incident with $j$ 
formed by two bridging arcs $\alpha$ and $\beta$ with endpoint $r$ on 
the inner boundary ($n+1\le r\le n+m$) 
and 
a peripheral arc or boundary segment with endpoint $j$ on the outer boundary. 
For brevity, we denote this arc or boundary segment $\gamma$. 
If this triangle 
lies to the left (resp.\ right) of $\alpha$, we insert a new vertex $v_1$ between $r-1$ and $r$ 
(resp.\ between $r$ and $r+1$), with the appropriate reduction of 
$r-1$ (or $r+1$) if it is not in $\{n+1,\dots, n+m\}$. 
We thus work in $A_{n,m+1}$. 
We attach $\beta$, together with all arcs incident with $r$ lying to the left (resp.\ right) of $\beta$, 
with the new vertex $v_1$. 
There remains a quadrilateral with sides $\alpha$, $\gamma$, $\beta$ and 
$[v_1,r]$ (resp.\ $[r,v_1]$).  
To obtain a triangulation as claimed, we add a new bridging arc connecting $j$ and $v_1$ 
and we are done. 

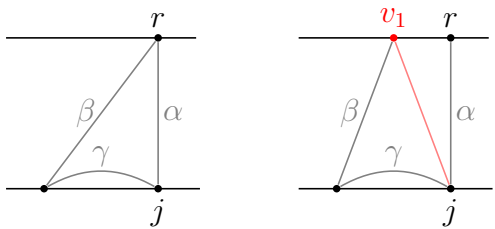
\begin{figure}[H]
\begin{tikzpicture}[scale=1,font=\normalsize] 

\draw[semithick] (-1,1) to (1.5,1);

\node (a) at (1,1) [fill,circle,inner sep=1pt] {};
\draw (a) node [above] {$r$};

\draw[semithick] (-1,-1) to (1.55,-1);

\node (1) at (-0.5,-1) [fill,circle,inner sep=1pt] {};
\node (j) at (1,-1) [fill,circle,inner sep=1pt] {};
\draw (j) node [below] {$j$};

\draw[semithick,opacity=0.5,out=30,in=150] (1) to node [shift={(0,0.2)}] {$\gamma$} (j);
\draw[semithick,opacity=0.5] (j) to node [shift={(0.2,0)}] {$\alpha$}  (a);
\draw[semithick,opacity=0.5] (1) to node [shift={(-0.2,0)}] {$\beta$} (a);

\end{tikzpicture}
\hskip 1cm
\begin{tikzpicture}[scale=1,font=\normalsize] 

\draw[semithick] (-1,1) to (1.5,1);

\node[red] (v) at (0.25,1) [fill,circle,inner sep=1pt] {};
\node (a) at (1,1) [fill,circle,inner sep=1pt] {};
\draw[red] (v) node [above] {$v_1$};
\draw (a) node [above] {$r$};

\draw[semithick] (-1,-1) to (1.55,-1);

\node (1) at (-0.5,-1) [fill,circle,inner sep=1pt] {};
\node (j) at (1,-1) [fill,circle,inner sep=1pt] {};
\draw (j) node [below] {$j$};

\draw[semithick,opacity=0.5,out=30,in=150] (1) to node [shift={(0,0.2)}] {$\gamma$} (j);
\draw[semithick,opacity=0.5] (j) to node [shift={(0.2,0)}] {$\alpha$}  (a);
\draw[semithick,opacity=0.5] (1) to node [shift={(-0.2,0)}] {$\beta$} (v);
\draw[semithick,red,opacity=0.5] (j) to (v);

\end{tikzpicture}
\end{figure}

\item If there is an asymptotic arc at $j$, 
then we can assume that $T$ is a triangulation 
of $A_{n,0}$. All 
asymptotic arcs are of the same type, so we may take adic arcs 
and $\alpha_j\in T$. 
Consider the fundamental domain of $\U$ on the lower boundary vertices 
$(j-1,0),(j,0),\dots, (j+n-2,0),(j+n-1,0)$, where both $(j-1,0)$ and $(j+n-1,0)$ correspond to $j$. 
We add a new vertex to the upper boundary and 
replace every asymptotic arc by a bridging arc 
from the corresponding vertex of the lower boundary to the new vertex of the upper 
boundary. Since there are bridging arcs from $(j-1,0)$ and $(j+n-1,0)$ to the new vertex, 
the number of triangles incident with $j$ (in the annulus) is $a_j+1$.

The figures illustrates this for $n=5$, $j=3$. 
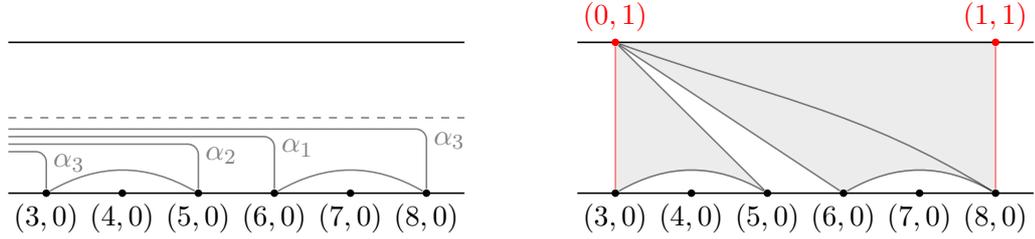
\begin{figure}[H]
\begin{tikzpicture}[scale=1,font=\small] 

\draw[semithick,opacity=0.5,dashed] (-0.5,0) to (5.5,0);

\draw[semithick] (-0.5,1) to (5.5,1);

\draw[semithick] (-0.5,-1) to (5.5,-1);

\foreach \x in {0,1,2,3,4,5} {
   \begin{scope}[shift={(\x cm, 0 cm)}]
    \node (\x) at (0,-1) [fill,circle,inner sep=1pt] {};
   \end{scope}
}
\draw (0) node [below] {$(3,0)$};
\draw (1) node [below] {$(4,0)$};
\draw (2) node [below] {$(5,0)$};
\draw (3) node [below] {$(6,0)$};
\draw (4) node [below] {$(7,0)$};
\draw (5) node [below] {$(8,0)$};

\draw[semithick,opacity=0.5,out=30,in=150] (0) to (2);
\draw[rounded corners,semithick,opacity=0.5] (0) to  (0,-0.45) to (-0.5,-0.45);
\draw[opacity=0.5] (0) node [shift={(0.3,0.4)}] {$\alpha_3$};
\draw[rounded corners,semithick,opacity=0.5] (2) to  (2,-0.35) to (-0.5,-0.35);
\draw[opacity=0.5] (2) node [shift={(0.3,0.5)}] {$\alpha_2$};
\draw[rounded corners,semithick,opacity=0.5] (3) to  (3,-0.25) to (-0.5,-0.25);
\draw[opacity=0.5] (3) node [shift={(0.3,0.6)}] {$\alpha_1$};
\draw[semithick,opacity=0.5,out=30,in=150] (3) to (5);
\draw[rounded corners,semithick,opacity=0.5] (5) to  (5,-0.15) to (-0.5,-0.15);
\draw[opacity=0.5] (5) node [shift={(0.3,0.7)}] {$\alpha_3$};

\end{tikzpicture}
\hskip 1cm
\begin{tikzpicture}[scale=1,font=\small] 

\draw[opacity=0,fill=gray,fill opacity=0.15] (0,-1) -- (0,1) --  (2,-1) to [out=150,in=30]  (0,-1);
\draw[opacity=0,fill=gray,fill opacity=0.15] (3,-1) -- (0,1) -- (5,1) --  (5,-1) to [out=150,in=30]  (3,-1);

\draw[semithick] (-0.5,1) to (5.5,1);

\node[red] (a) at (0,1) [fill,circle,inner sep=1pt] {};
\node[red] (b) at (5,1) [fill,circle,inner sep=1pt] {};

\draw[red] (a) node [above] {$(0,1)$};
\draw[red] (b) node [above] {$(1,1)$};

\draw[semithick] (-0.5,-1) to (5.5,-1);

\foreach \x in {0,1,2,3,4,5} {
   \begin{scope}[shift={(\x cm, 0 cm)}]
    \node (\x) at (0,-1) [fill,circle,inner sep=1pt] {};
   \end{scope}
}

\draw (0) node [below] {$(3,0)$};
\draw (1) node [below] {$(4,0)$};
\draw (2) node [below] {$(5,0)$};
\draw (3) node [below] {$(6,0)$};
\draw (4) node [below] {$(7,0)$};
\draw (5) node [below] {$(8,0)$};


\draw[semithick,opacity=0.5,out=30,in=150] (0) to (2);
\draw[semithick,opacity=0.5,out=30,in=150] (3) to (5);

\draw[semithick,red,opacity=0.5] (0) to (a);
\draw[semithick,opacity=0.5] (2) to (a);
\draw[semithick,opacity=0.5] (3) to (a);
\draw[semithick,opacity=0.5,out=148,in=-20] (5) to (a);
\draw[semithick,red,opacity=0.5] (5) to (b);

\end{tikzpicture}
\end{figure}

\end{inparaenum}
\end{inparaenum}
\end{proof}

Note that when increasing an entry of a quiddity sequence, the 
result may have smaller period and thus be realizable in an annulus with fewer points on 
both boundaries. The following example illustrates this. 

\begin{ex}
Consider the quiddity sequence $(4,1,5,1)$. It is realizable by a triangulation of $A_{4,1}$ with 
two peripheral arcs $[1,3]$, $[3,1]$, one bridging arc from $1$ to $5$ and 
two bridging arcs from $3$ to $5$. There is one bridging arc at $1$, so following 
the arguments of the proof of the theorem (case  B(ii)), 
$(5,1,5,1)$ is realizable by $A_{4,2}$. 
The triangulation constructed in the proof of Proposition~\ref{prop:adding-one} 
has the two peripheral arcs $[1,3]$, $[3,1]$, and the four bridging arcs 
$[1,5]$, $[1,v_1]$, $[3,v_1]$ and $[3,5]$. 
However, $(5,1,5,1)$ is 2-periodic, and can be 
realized by a triangulation of $A_{2,1}$ with one peripheral arc $[1,1]$ and 
two bridging arcs connecting $1$ with $3$. 
\end{ex}

If we take a triangulation of $A_{n,0}$ with $n$ adic arcs, the associated 
sequence is the constant sequence with $a_i=2$ for all $i$. 
We already know that this gives an infinite frieze, cf.\ 
Figure~\ref{fig:basicfrieze}. Using this, an immediate consequence 
of the theorem is that any sequence where every $a_i$ is at least $2$ 
is realizable. 
In this case, we can actually say much more, as we can construct a triangulation 
directly: 

\begin{cor}\label{cor:bounding-realizable}
Every periodic infinite frieze with quiddity sequence $(a_1,\dots, a_n)$ 
where $a_i\ge 2$ for all $i$ is realizable. Furthermore, if $a_i>2$ for some $i$, it
can be realized in a triangulation of an annulus which contains only bridging arcs. 
\end{cor}

\begin{proof}
If $a_i=2$ for all $i$, we can take the annulus $A_{1,0}$ and triangulate it 
with one asymptotic arc. 
To construct a triangulation in the other cases, we can use a similar 
strategy as in 
the proof of Proposition~\ref{prop:adding-one}. Set $d_i:=a_i-2$ for all $i$, and $m:=\sum d_i$. 
We draw an annulus $A_{n,0}$ as a cylinder in the plane, with 
marked points $1,2,\dots, n,1$ 
on the lower boundary. 
Above each marked point $i$, draw $d_i$ marked points on the upper boundary 
and connect $i$ with these $d_i$ marked points using bridging arcs. 
The result of this is a set of pairwise compatible arcs in $A_{n,m}$. 
We then complete 
this set of arcs to a triangulation by adding one further (compatible) 
bridging arc for each marked point 
on the lower boundary. 
We thereby obtain a triangulation of $A_{n,m}$ 
consisting solely of bridging arcs. By construction, each 
$i$ on the lower boundary has $a_i-1$ arcs to the upper boundary and hence 
is incident with $a_i$ triangles. 
\end{proof}

\begin{theorem}\label{thm:periodic-combinatorial}
Let $\mathcal F$ be a periodic frieze with quiddity sequence $(a_1,\dots, a_n)$. 
Then either $\mathcal F$ is finite or $(a_1,\dots, a_n)$ is realizable in $A_{n,m}$ 
for some $m\ge 0$. 
\end{theorem}

\begin{proof}
Let $\mathcal F$ be $n$-periodic with quiddity sequence $(a_1,\dots, a_n)$, 
and assume that $n$ is its shortest period. 

We first consider friezes of period at most $2$, making use of 
Example~\ref{ex:small-period}: \\
Let $\mathcal F$ have period $1$. If $(a_1)=(1)$, the frieze is necessarily 
finite, with 4 rows, coming from the trivial triangulation of the triangle.
If the quiddity sequence is $(a)$, with $a\ge 2$, then $\mathcal F$ is infinite. 
The claim follows from Corollary~\ref{cor:bounding-realizable}. 

Let $\mathcal F$ have period $2$. If $(a_1,a_2)=(1,2)$ or $(a_1,a_2)=(1,3)$, 
then $\mathcal F$ is finite, coming from a triangulation of a quadrilateral or 
from a hexagon triangulated by an inner triangle. If $(a_1,a_2)=(1,a)$ with $a\ge 4$, $\mathcal F$ is infinite. We have seen 
in Example~\ref{ex:small-period} (d) and (e) that this sequence is realizable. 
The case $(a_1,a_2)$ with $a_i\ge 2$ is again infinite and 
covered in Corollary~\ref{cor:bounding-realizable}. 

We now assume that $\mathcal F$ has period $n\ge 3$. Let 
$q_0:=(a_1,\dots, a_n)$. If all entries of $q_0$ are at least
$2$, $\mathcal F$ is infinite and realizable in $A_{n,m}$ for some $m$.
If $(a_1,\dots, a_n)$ contains an entry $1$, we reduce $q_0$ to a sequence 
with $n-1$ entries. At this point, we cannot have two entries $a_i=a_{i+1}=1$, since then, the 
unimodular rule would imply $m_{i,i+1}=0$. This is only possible in the finite frieze 
with $a_i=1$ for all $i$, a $1$-periodic frieze, contradicting the assumption 
on the period. 

W.l.o.g.\ let $a_n=1$. 
Consider the sequence $q_1:=(a_1-1,a_2,\dots, a_{n-2},a_{n-1}-1)$ obtained by subtracting 
$1$ from $a_1$ and from $a_{n-1}$ and by dropping its $n$th entry $a_n$. 
If $\mathcal F$ is finite, the new sequence defines a finite frieze by \cite[Question 23]{coco1,coco2}; 
if $\mathcal F$ is infinite, the new sequence defines an infinite frieze by \cite[Theorem 3.7]{tschabold}. 
Its shortest period is $n-1$ or a smaller 
divisor of $n-1$. 

\begin{inparaenum}[(i)]
\item If all entries of $q_1$ are at least $2$,  
$q_1$ is realizable in $A_{n-1,m}$ with an appropriate $m$. 
Let $T_1$ be a triangulation of $A_{n-1,m}$ realizing $q_1$. By adding a new marked point $n$ between $n-1$ and $1$
on the outer boundary (relabeling the marked points on the inner boundary appropriately), and adding the peripheral
arc $[n-1,1]$, we obtain a triangulation of $A_{n,m}$ realizing $q_0$ and we are done. 

\item So assume that $q_1$ still has a $1$ among its entries. 
If $q_1$ is $2$-periodic, say 
$q_1=(1,a,\dots, 1,a)$ with $2\cdot s=n-1$ for $s\ge 1$, the frieze is finite for $a=2,3$. 
If $a\ge 4$, 
$\mathcal F$ is infinite. 
In this case, we can realize $(1,a)$ in an annulus $A_{2,m'}$ 
and thus $q_1$ in $A_{n-1,m}$ for $m=s\cdot m'$ by Lemma~\ref{lm:period-multiply}. 
To realize $q_0$, we proceed as above. We pick a triangulation of $A_{n-1,m}$ realizing 
$q_1$, insert a new marked point $n$ on the outer boundary between $n-1$ and $1$ (relabeling 
the marked points on the inner boundary), 
and add the peripheral arc $[n-1,1]$. 
The resulting triangulation of $A_{n,m}$ realizes $q_0$. 

If the shortest period of $q_1$ is still at least $3$, we iterate the reduction procedure. After finitely 
many, say $t$, reductions, the sequence $q_t$ will either be of the form $(1,2)$ or $(1,3)$ 
and hence define a finite frieze or it will be realizable in 
$A_{n-r,m'}$ with $m'$ appropriate. To find a triangulation of $A_{n,m}$ realizing $q_0$, we can work 
backwards, adding marked points to the outer boundary iteratedly.
\end{inparaenum}
\end{proof}

\begin{remark}
Note that the proof of Theorem~\ref{thm:periodic-combinatorial} provides an algorithm 
for testing whether a sequence $(a_1,\dots, a_n)$ gives rise to a frieze 
(and moreover, whether it is finite or infinite) or not: If the sequence has shortest period 
at most 2, we can say immediately. So, suppose that the sequence has shortest period $n \geq 3$.
If two consecutive entries (in the cyclic order) are 1, then it doesn't give a frieze. If all
the entries are at least 2, it defines an infinite frieze. Otherwise, say if $a_n=1$, we reduce
the sequence to $(a_1-1,a_2,\dots, a_{n-2},a_{n-1}-1)$. We then test this new sequence, noting 
that the answer must match that for the original sequence. Since the shortest period decreases
with each reduction, we are guaranteed a conclusive outcome after a finite number of steps.
\end{remark}

\begin{remark}\label{rm:optimal}
Let $\mathcal{F}$ be a periodic infinite frieze with shortest period $n$ and quiddity sequence 
$(a_1, \ldots, a_n)$. From above, $(a_1, \ldots, a_n)$ can be realized in an annulus 
$A_{n,m}$ for some $m$. We may also readily deduce the minimal value of $m$ 
for which this is possible (corresponding to a triangulation having no peripheral 
arcs on the inner boundary). 
Upon applying the algorithm to $(a_1, \ldots, a_n)$, suppose that the sequence obtained 
at the point at which we reach a stopping condition is $(b_1, \ldots, b_r)$, with 
$r$ the shortest period. 
We must have either (i) $r=1$ and $b_i \geq 2$; (ii) $r=2$ with $b_1 = 1$ and 
$b_2 \geq 4$ (or vice versa); or (iii) $r \geq 2$ and $b_i \geq 2$ for all $i$, 
$1 \leq i \leq r$. In these cases, the minimal value for $m$ is respectively $b_1 -2$, $b_2 -4$ 
and $\sum_{i=1}^r (b_i -2)$.
\end{remark}

%
\section{Matchings and a characterization of infinite friezes}\label{sec:final-matchings}
%

We now relax our set-up, considering triangulations of an infinite strip in the plane 
which need not be periodic. These triangulations are shown to also give rise to infinite 
friezes. 
In turn, we establish 
that all infinite friezes arise in this manner. 
Furthermore, we show that 
the entries in an infinite frieze 
are matching numbers obtained from matchings between vertices and triangles in any realization.

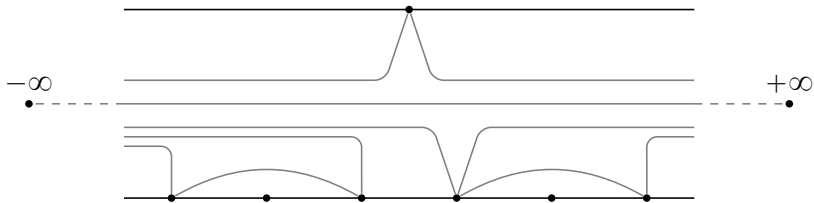
\begin{figure}[H]
\begin{tikzpicture}[scale=1.25,font=\small] 

 \node (a) at (-1.5,0) [fill,circle,inner sep=1pt] {};
  \node (b) at (6.5,0) [fill,circle,inner sep=1pt] {};
  
 \draw (a) node [above] {$- \infty$};
 \draw (b) node [above] {$+ \infty$};

\draw[semithick,opacity=0.5] (-0.5,0) to (5.5,0);
\draw[semithick,opacity=0.5,dashed] (-0.5,0) to (-1.5,0);
\draw[semithick,opacity=0.5,dashed] (5.5,0) to (6.5,0);

\draw[semithick] (-0.5,1) to (5.5,1);
\node (a) at (2.5,1) [fill,circle,inner sep=1pt] {};

\draw[rounded corners,semithick,opacity=0.5] (a) to (2.25,0.25) to (-0.5,0.25);
\draw[rounded corners,semithick,opacity=0.5] (a) to  (2.75,0.25) to (5.5,0.25);

\draw[semithick] (-0.5,-1) to (5.5,-1);

\foreach \x in {0,1,2,3,4,5} {
   \begin{scope}[shift={(\x cm, 0 cm)}]
    \node (\x) at (0,-1) [fill,circle,inner sep=1pt] {};
   \end{scope}
}

\draw[semithick,opacity=0.5,out=30,in=150] (0) to (2);
\draw[rounded corners,semithick,opacity=0.5] (0) to  (0,-0.45) to (-0.5,-0.45);
\draw[rounded corners,semithick,opacity=0.5] (2) to  (2,-0.35) to (-0.5,-0.35);
\draw[rounded corners,semithick,opacity=0.5] (3) to  (2.75,-0.25) to (-0.5,-0.25);
\draw[semithick,opacity=0.5,out=30,in=150] (3) to (5);
\draw[rounded corners,semithick,opacity=0.5] (5) to  (5,-0.35) to (5.5,-0.35);
\draw[rounded corners,semithick,opacity=0.5] (3) to  (3.25,-0.25) to (5.5,-0.25);

\end{tikzpicture}
\end{figure}

We will use $\V=\V(M_1,M_2)$ to denote an infinite strip of height one in the plane, 
with lower boundary $\{(x,0)\mid x\in\mathbb{R}\}$, upper boundary 
$\{(x,1)\mid x\in\mathbb{R}\}$, together with vertices $M_1:=\{(i,0) \mid i\in\ZZ\}$ on the lower boundary 
and a set $M_2=\{(q,1)\mid q\in A\subseteq\mathbb{Q}\}$ of vertices on the upper boundary. 
In addition, we have two limit points $\pm\infty$ appearing respectively at the right and left extremities of this strip
(see the figure above).
For convenience, we sometimes use $i$ to denote $(i,0)$. 

We call an arc between marked points {\em bridging} (resp.\ {\em peripheral}) 
if its endpoints belong to different boundaries (resp.\ the same boundary) of $\V$.  
The curve connecting the two 
limit points is called the {\em generic arc}.
An arc is said to be {\em asymptotic}, if it 
starts at a marked point on a boundary and tends towards $+\infty$ or $-\infty$. 

We want to consider (infinite) triangulations of 
$\V$. These are maximal collections of pairwise compatible arcs 
between vertices. 
Note that if a triangulation contains no bridging arc, it must contain the generic arc 
as this is compatible with all peripheral and asymptotic arcs. 
In particular,
we are interested in those triangulations for which all vertices
on the lower boundary are incident with only finitely many arcs.
We refer to such a triangulation as an \emph{admissible triangulation} of $\V$. 

\begin{defn}
We say that an infinite sequence $(a_i)_i$ of positive 
integers (or a frieze $\mathcal F$ with quiddity row $(a_i)_i$) is 
{\em realizable in $\V$} if there 
exists a set $M_2:=\{(q,1)\mid q\in A\subseteq\mathbb{Q}\}$ of marked points 
on the upper boundary and a (admissible) triangulation of $\V(M_1,M_2)$ such 
that $a_i$ is equal to the number of triangles 
incident with $(i,0)$ for every $i\in \mathbb{Z}$.  
\end{defn}

In view of Theorem~\ref{thm:periodic-combinatorial}, 
it is natural to ask if every infinite frieze can be realized in $\mathbb{V}$. 
We next show that this is indeed the case, 
giving a constructive proof. 

\begin{theorem}\label{thm:friezes-realize}
Every infinite frieze is realizable in $\V$. 
\end{theorem}

\begin{proof}
Let $q^{(0)} = (a_i^{(0)})_{i \in \mathbb{Z}}$ be the quiddity row of an infinite frieze. We explicitly construct a triangulation 
of $\mathbb{V}$ realizing $q^{(0)}$. If $a_i^{(0)} \ge 2$ for all $i$, then the construction can be done using essentially the 
same strategy as the one described in Corollary~\ref{cor:bounding-realizable}. So, suppose that some of the entries 
of $q^{(0)}$ are 1's. Indeed, let $Z_0 = \{ t \in \mathbb{Z} \mid a_t^{(0)} = 1 \}$, and recall that no two consecutive 
entries of $q^{(0)}$ can both be 1.

We start from the infinite strip $\mathbb{V} = \mathbb{V}(M_1,\emptyset)$, not yet containing any arcs. For each 
$t \in Z_0$, we add the peripheral arc $[t-1,t+1]$ in $\mathbb{V}$. The peripheral arcs added in this step are pairwise 
compatible. 
We also reduce $q^{(0)}$ at each 1, giving the new quiddity row $q^{(1)} = ( a_i^{(1)} )_{i \in \mathbb{Z} \setminus Z_0}$ 
where $a_{i+1}^{(1)}=a_{i+1}^{(0)}-1$ if $i \in Z_0$ and $i+2 \notin Z_0$ (similarly for $a_{i-1}^{(1)}$), 
$a_{i+1}^{(1)}=a_{i+1}^{(0)}-2$ if $i,i+2 \in Z_0$ (similarly for $a_{i-1}^{(1)}$), and 
$a_i^{(1)} = a_i^{(0)}$ otherwise. 

If $q^{(1)}$ has 1's, we repeat the process of adding arcs and reducing. In general, let $q^{(r)} = (a_i^{(r)})_{i \in \mathbb{Z} \setminus Z_{r-1}}$ 
be the quiddity row obtained after $r$ steps, with $Z_r = Z_{r-1} \cup \{ t \in \mathbb{Z} \setminus Z_{r-1} \mid a_t^{(r)}=1 \}$. 
For each $t \in Z_r \setminus Z_{r-1}$, we then add the peripheral arc $[t_-^{(r)},t_+^{(r)}]$ in $\mathbb{V}$ (above any previously 
added peripheral arcs), where $t_-^{(r)}$ (resp.\ $t_+^{(r)}$) is the largest (resp.\ smallest) value below (resp.\ above) $t$ in 
$\mathbb{Z} \setminus Z_{r-1}$. These peripheral arcs are pairwise compatible and also compatible with all previously introduced arcs.

If at some point, we reach a quiddity row $q^{(r)}$ having no 1's (i.e.\ if $Z_r = Z_{r-1}$), then we can complete our 
collection of arcs to a triangulation of $\mathbb{V} = \mathbb{V}(M_1,M_2)$ realizing $q^{(0)}$, using bridging arcs 
and an appropriate choice of $M_2$.

So, assume $Z_r \supsetneq Z_{r-1}$ for all $r \ge 1$. Denote by $C_r$ the collection of (peripheral) arcs added in the $r$th step. 
Let $C^{(r)} = \cup_{j=1}^r C_j$ be the collection of all arcs added up to and including the $r$th step, and let 
$C^{\infty} = \left ( \cup_{j=1}^{\infty} C_j \right ) \cup [-\infty,\infty]$ be the total collection of all arcs we obtain via our procedure, 
together with the generic arc. We will show that $C^{\infty}$ is a triangulation of $\mathbb{V}(M_1,\emptyset)$. It realizes $q^{(0)}$ by construction.

Since $Z_r \supsetneq Z_{r-1}$ for all $r \ge 1$, it may be routinely established that for any given $i \in \mathbb{Z}$, 
there exists some $N \ge 0$ such that $a_i^{(N)}=1$. It follows immediately that the asymptotic arcs incident 
with $i$ are not compatible with $C_{N+1}$ (and hence also $C^{(N+1)}, C^{\infty}$). Note also that there exist no 
bridging arcs in $\mathbb{V}(M_1,\emptyset)$. In order to check that $C^{\infty}$ is a maximal collection of pairwise 
compatible arcs, it thus remains to check that an arbitrary peripheral arc either belongs to $C^{\infty}$ or crosses an 
arc belonging to $C^{\infty}$. Let $u,v \in \mathbb{Z}$, $u<v-1$ and consider the peripheral arc $[u,v]$. We distinguish three separate cases.

\begin{inparaenum}[\text{Case} $1$:] 
\item 
If $a_u^{(0)}=1$ or $a_v^{(0)}=1$, say $a_u^{(0)}=1$, then $[u-1,u+1]\in C_1$ and  
$[u,v]$ crosses this arc. 

\item 
Suppose $a_i^{(0)} \neq 1$ for all $u \le i \le v$. We choose $r$ to be minimal such that $a_u^{(r)}=1$ or $a_v^{(r)}=1$. 
Say $a_u^{(r)}=1$. Then $[u_-^{(r)},u+1] \in C_{r+1}$, and $[u,v]$ crosses this arc.

\item  
There exists a 1 (strictly) between $a_u^{(0)}$ and $a_v^{(0)}$ in $q^{(0)}$. We again choose $r$ to be minimal such that 
$a_u^{(r)}=1$ or $a_v^{(r)}=1$. If $a_u^{(r)}=a_v^{(r)}=1$, we must have $u_+^{(r)}<v$ and hence $[u,v]$ crosses the arc $[u_-^{(r)},u_+^{(r)}] \in C_{r+1}$. 
So, assume $a_u^{(r)}=1$ and $a_v^{(r)} \ge 2$. We now either have 
$u_+^{(r)}<v$ in which case $[u,v]$ crosses the arc $[u_-^{(r)},u_+^{(r)}] \in C_{r+1}$, or $u_+^{(r)}=v$ in which case $[u,v] \in C^{(r)}$.
\end{inparaenum}
\end{proof}

It is known that every entry in a finite frieze is given 
as a matching number between vertices of a polygon and triangles 
of the associated triangulation, see \cite{bci}. 
For periodic infinite friezes arising from triangulations of punctured 
discs, the entries $m_{ij}$ are also matching numbers, see \cite[Theorem~5.21]{tschabold}.

We are now able to extend this by showing that any admissible triangulation of $\V$ gives rise 
to an infinite frieze whose entries are the matching numbers for the triangulation, 
and thus that the frieze is determined completely by the geometry of the triangulation.
By Theorem~\ref{thm:friezes-realize}, this covers 
all infinite friezes 
(including in particular, periodic infinite friezes). 

\begin{defn}
Let $T$ be an admissible triangulation of $\V$. Let $i=(i,0)$ be 
a vertex on the lower boundary and $b\ge 0$. Then 
$\mathcal M_{i,i+b}$ $= \mathcal M_{i,i+b}(T)$ is the set of matchings between the vertices 
$\{(i,0),(i+1,0),\dots, (i+b,0)\}$ and triangles in $T$. 
\end{defn}

We refer to $\abs{\mathcal M_{ij}}$, for $i\le j$, as the \emph{$(i,j)$-matching number}. 
By convention, we take the $(i,i-1)$-matching number and the $(i,i-2)$-matching number 
to be $1$ and $0$, respectively, for all $i$. 

The proof of the following result is straightforward and thus we omit it. 

\begin{lm}\label{lm:matchingrelation}
Let $T$ be an admissible triangulation of $\mathbb{V}$ with no 
peripheral arcs on the lower boundary
and let $b \ge 0$. Using $a_i$ to denote the number of triangles incident with the vertex 
$(i,0)$, for all $i$,we have 
$$
\abs{\mathcal{M}_{i,i+b}}=(a_i-1) \abs{\mathcal{M}_{i+1,i+b}}+
\sum_{k=i+1}^{i+b-1}\left( (a_k-2)  \abs{\mathcal{M}_{k+1,i+b}}\right)+ a_{i+b}-1.
$$
\end{lm}

\begin{prop}\label{prop:entries-matching}
Let $T$ be an admissible triangulation of $\mathbb{V}$ with no peripheral arcs on the lower boundary,
and let $a_i$ be the number of triangles incident with $(i,0)$, for every $i$. Then $(a_i)_i$ is the 
quiddity row of an infinite frieze $\mathcal{F}=(m_{ij})_{i,j}$, and $m_{ij}=\abs{\mathcal{M}_{ij}}$
for all $i \le j$. 
\end{prop}

\begin{proof}
That $(a_i)_i$ is the quiddity row of an infinite frieze 
is immediate by Corollary~\ref{cor:all-2-ok}, since $a_i \ge 2$ for all $i$.

Let $i,j$ be marked points on the lower boundary with $i\le j$. 
We show that $m_{ij}=\abs{\mathcal{M}_{ij}}$ using 
induction on $j-i$.

For $i=j$, the claim is clear. 

Now suppose $j-i= r$ and consider 
$\mathcal{M}_{i,j+1}$. By Lemma \ref{lm:matchingrelation}, 
$$
\abs{\mathcal{M}_{i,j+1}}=(a_i-1) \abs{\mathcal{M}_{i+1,j+1}} + 
\sum_{k=i+1}^j \left ( (a_k-2)  \abs{\mathcal{M}_{k+1,j+1}} \right ) + a_{j+1}-1.
$$
All sets of matchings on the right hand side involve pairs of vertices 
$(l,j+1)$ with $j+1-l\le r$, and so we can use induction to replace the 
corresponding numbers by the $m_{l,j+1}$: 
$$
\abs{\mathcal{M}_{i,j+1}}=(a_i-1)m_{i+1,j+1}+\sum_{k=i+1}^j \left ( (a_k-2)  m_{k+1,j+1} \right ) +a_{j+1}-1.
$$
We rewrite this and use (\ref{friezerelationb}): 
\begin{align*}
\abs{\mathcal{M}_{i,j+1}} & = \ \overbrace{a_im_{i+1,j+1}-m_{i+2,j+1}}^{m_{i,j+1}} - m_{i+1,j+1} \\ 
 & \quad + \overbrace{a_{i+1}m_{i+2,j+1}-m_{i+3,j+1}}^{m_{i+1,j+1}} - m_{i+2,j+1} \\ 
 & \quad + \quad \quad \vdots \\
 & \quad + \overbrace{a_{j-1}m_{j,j+1}-m_{j+1,j+1}}^{m_{j-1,j+1}} - m_{j,j+1} \\  
 & \quad +  a_jm_{j+1,j+1}  - m_{j+1,j+1} + a_{j+1} -1 
\end{align*} 
Most terms cancel, leaving $\abs{\mathcal{M}_{i,j+1}}=m_{i,j+1}+a_jm_{j+1,j+1}-m_{j,j+1}-1= m_{i,j+1}$, 
with the latter equality following from the unimodular rule.
\end{proof}

We now extend this 
to arbitrary admissible triangulations of $\V$.

\begin{theorem}\label{thm:arbitrary-matchings} 
Let $T$ be an admissible triangulation of $\V$ and let 
$a_i$ be the number of triangles incident with $(i,0)$, for every $i$. Then 
$(a_i)_i$ is the quiddity row of an infinite frieze  $\mathcal F=(m_{ij})_{i,j}$.
Moreover, 
$
m_{ij}=\abs{\mathcal{M}_{ij}},
$
for all $i \le j$. \\
\end{theorem}
Note that by Theorem~\ref{thm:friezes-realize} any infinite frieze is realizable and thus 
the entries of such 
a frieze are matching numbers for any triangulation realizing it. 

\medskip

Before commencing the proof, we first note that our task can be simplified somewhat: 
Let $T$ be an admissible triangulation of $\V(M_1,M_2)$, let $i\le j$. 
We have that $\mathcal M_{ij}$ only depends on the set $\Delta_T(i, j)$ of triangles 
incident with a vertex in $\{(i,0),\dots, (j,0)\}$. As a consequence of the following lemma, for which 
the proof is straightforward, we see that it is sufficient to prove Theorem~\ref{thm:arbitrary-matchings} 
for admissible triangulations having only finitely many peripheral arcs and 
no asymptotic arcs. 

\begin{lm}\label{lm:peripheral-finite}
Let $T$ be an admissible triangulation of $\V(M_1,M_2)$, and let $i$, $j$ with $i\le j$. 
Then there exists an admissible triangulation $T'$ of $\V(M_1,M_2')$, where $T'$ has only 
finitely many peripheral arcs and no asymptotic arcs, such that there is a one-to-one correspondence 
between $\Delta_T(i,j)$ and $\Delta_{T'}(i,j)$ preserving incidences with the vertices $\{(i,0),\dots, (j,0)\}$. 
\end{lm}

\begin{proof}[Proof of Theorem~\ref{thm:arbitrary-matchings}]

We prove the result by induction on the number $t$ of peripheral arcs of $T$ on the lower 
boundary. 

If $t=0$, the result follows from 
Proposition~\ref{prop:entries-matching}. 

Let $t\ge 0$ and assume the result holds for every admissible triangulation of $\V$ 
containing $t$ peripheral arcs on the lower boundary. We consider a triangulation $T$ 
of $\V$ with $t+1$ peripheral arcs on the lower boundary. There exists a peripheral 
triangle 
at some vertex $(k,0)$, and $a_k=1$. Removing this triangle 
leads to a 
triangulation $\widetilde T$ of $\V$ (relabeling the vertices on the lower boundary on the 
right of $k-1$ appropriately), with $t$ peripheral arcs, providing the sequence 
$(\dots, a_{k-2}, a_{k-1}-1,a_{k+1}-1,a_{k+2},\dots)$. 
By induction, this sequence 
is the quiddity row of an infinite frieze $\widetilde{\mathcal F}=(\tilde m_{ij})_{i,j}$ 
such that the entries $\tilde m_{ij}$ are the matching numbers for the 
triangulation $\widetilde T$. 

One can routinely check that the following holds (cf.  \cite[Theorem 3.1]{tschabold}): 
$$
m_{ij}=
\begin{cases}
\tilde m_{\tilde{i}+1,\tilde{j}}+\tilde m_{\tilde{i},\tilde{j}} & i=k+1,\\
\tilde m_{\tilde{i},\tilde{j}-1}+\tilde m_{\tilde{i},\tilde{j}} &j=k-1,\\
\tilde m_{\tilde{i},\tilde{j}} & \mbox{otherwise,} 
\end{cases}
$$
where $\tilde{i}=i$ if $i\le k$, $\tilde{i}=i-1$ for $i>k$, $\tilde{j}=j$ for $j<k$ and 
$\tilde{j}=j-1$ for $j\ge k$. 
That the matching number $\abs{\mathcal M_{ij}}$ is given in the same way, can be 
checked by direct computation. 
\end{proof}

\begin{cor}
Let $\mathcal{F}=(m_{ij})_{i,j}$ be a periodic infinite frieze and let $T$ be a triangulation of an annulus 
realizing its quiddity sequence. Considering $T$ in the universal cover, we have that the 
entries $m_{ij}$, for $i\le j$, are matching numbers between 
marked points on the lower boundary and triangles. 
\end{cor}

\begin{remark}\label{rm:minimal-no-unused}
Let $\mathcal{F}$ be an infinite frieze. In any triangulation of the infinite strip which realizes 
$\mathcal{F}$ and has no peripheral or asymptotic arcs 
on the upper boundary, there are no unused 
triangles when computing matching numbers (as all triangles are incident with a vertex on the 
lower boundary).
\end{remark}

%
\subsection*{Acknowledgements}
%

The first author was supported by the Austrian Science Fund projects 
FWF P25141-N26 and FWF W1230. 
She acknowledges support by 
the Institute Mittag-Leffler, Djursholm, Sweden.  
The second author was supported by the Austrian Science Fund (FWF): 
Project No.\ P25141-N26. 
In addition, all authors 
acknowledge support from  NAWI Graz.


\end{document}